\documentclass[letterpaper, 11pt,  reqno]{amsart}

\usepackage{amsmath,amssymb,amscd,amsthm,amsxtra, esint}

\usepackage{hyperref} 

\allowdisplaybreaks[2]

\newtheorem{theorem}{Theorem} [section]

\newtheorem{lemma}[theorem]{Lemma}
\newtheorem{proposition}[theorem]{Proposition}
\newtheorem{remark}[theorem]{Remark}

\newtheorem{corollary}[theorem]{Corollary}

\DeclareMathOperator*{\intt}{\int}

\DeclareMathOperator*{\supp}{supp}

\newcommand{\I}{\hspace{0.5mm}\text{I}\hspace{0.5mm}}
\newcommand{\II}{\text{I \hspace{-2.8mm} I} }
\newcommand{\III}{\text{I \hspace{-2.9mm} I \hspace{-2.9mm} I}}

\newcommand{\IV}{\text{I \hspace{-2.9mm} V}}

\newcommand{\noi}{\noindent}
\newcommand{\Z}{\mathbb{Z}}
\newcommand{\R}{\mathbb{R}}
\newcommand{\C}{\mathbb{C}}
\newcommand{\T}{\mathbb{T}}

\let\Re=\undefined\DeclareMathOperator*{\Re}{Re}
\let\Im=\undefined\DeclareMathOperator*{\Im}{Im}

\let\P= \undefined
\newcommand{\P}{\mathbf{P}}

\renewcommand{\L}{\mathcal{L}}

\newcommand{\RR}{\mathcal{R}}

\newcommand{\F}{\mathcal{F}}

\newcommand{\al}{\alpha}
\newcommand{\be}{\beta}
\newcommand{\dl}{\delta}

\newcommand{\Dl}{\Delta}
\newcommand{\eps}{\varepsilon}

\newcommand{\g}{\gamma}
\newcommand{\G}{\Gamma}
\newcommand{\ld}{\lambda}

\newcommand{\s}{\sigma}

\newcommand{\ft}{\widehat}

\newcommand{\wt}{\widetilde}
\newcommand{\cj}{\overline}
\newcommand{\dx}{\partial_x}
\newcommand{\dt}{\partial_t}
\newcommand{\dd}{\partial}

\renewcommand{\l}{\ell}
\renewcommand{\o}{\omega}
\renewcommand{\O}{\Omega}

\newcommand{\les}{\lesssim}
\newcommand{\ges}{\gtrsim}

\newcommand{\jb}[1]
{\langle #1 \rangle}

\newcommand{\ind}{\mathbf 1}

\renewcommand{\S}{\mathcal{S}}

\newcommand{\N}{\mathbb{N}}

\newcommand{\NN}{\mathcal{N}}
\newcommand{\QQ}{\mathcal{Q}}

\newcommand{\fN}{\mathfrak{N}}
\newcommand{\fR}{\mathfrak{R}}

\newcommand{\GG}{\mathcal{G}}

\newtheorem*{ackno}{Acknowledgements}

\numberwithin{equation}{section}
\numberwithin{theorem}{section}

\begin{document}
\baselineskip = 14pt

\title[Quasi-invariant Gaussian measures for the cubic fourth order NLS]
{Quasi-invariant Gaussian measures 
for the cubic  fourth order 
nonlinear Schr\"odinger equation}

\author[T.~Oh and  N.~Tzvetkov]
{Tadahiro Oh and Nikolay Tzvetkov}

\address{
Tadahiro Oh, School of Mathematics\\
The University of Edinburgh\\
and The Maxwell Institute for the Mathematical Sciences\\
James Clerk Maxwell Building\\
The King's Buildings\\
Peter Guthrie Tait Road\\
Edinburgh\\ 
EH9 3FD\\
 United Kingdom}

\email{hiro.oh@ed.ac.uk}

\address{
Nikolay Tzvetkov\\
Universit\'e de Cergy-Pontoise\\
 2, av.~Adolphe Chauvin\\
  95302 Cergy-Pontoise Cedex \\
  France}

\email{nikolay.tzvetkov@u-cergy.fr}



\begin{abstract}
We consider the  cubic fourth order  nonlinear Schr\"odinger equation on the circle.
In particular, 
we prove that the mean-zero Gaussian measures on Sobolev spaces 
$H^s(\T)$, $s > \frac3 4$, 
are quasi-invariant under the flow.

\end{abstract}


\maketitle
\tableofcontents

\newpage
\section{Introduction}

\subsection{Background}

In this paper, we 
continue 
the  program set up by the second author \cite{TzBBM}
and 
study the transport property 
of Gaussian measures on Sobolev spaces
under the dynamics of a certain Hamiltonian partial differential equation (PDE).

In probability theory, 
there is an extensive literature
on the transport property 
of Gaussian measures 
under linear and nonlinear transformations.
See, for example,  \cite{CM, Kuo, RA, Cru1, Cru2, Bog, AF}.
Classically, Cameron-Martin \cite{CM} studied
 the transport property 
of Gaussian measures under a shift
and established a dichotomy
between absolute continuity and singularity
of the transported measure. 
In the context of nonlinear transformations, 
the  work in \cite{Kuo, RA}
considers nonlinear transformations
that are close to the identity, 
while the work in \cite{Cru1, Cru2}
considers the transport property
under the flow generated by (non-smooth) vector fields.
In particular, in \cite{Cru2}, the existence
of quasi-invariant measures under the dynamics
was established under 
an exponential integrability assumption
of the divergence of the corresponding vector field.
We also note a recent work 
\cite{NRSS} 
establishing absolute continuity 
of the Gaussian measure associated to the complex Brownian bridge
on the circle
under certain gauge transformations.

In the field of Hamiltonian PDEs, 
Gaussian measures
naturally appear in the construction 
of invariant measures 
associated to conservation laws
such as Gibbs measures.
These invariant measures associated to conservation laws
are typically constructed
as weighted Gaussian measures.
There has been a significant progress 
over the recent years in this subject.
See \cite{LRS, BO94, MV, McKean, BO96, BO97, Zhid, Zhid2, 
TZ1, TZ2, 
BTIMRN, BT2,  QV, 
OH3, OH4, OHSBO, 
Tzv, TTz, NORS, 
SuzzoniNLW, SuzzoniBBM, Deng, 
BTT, 
TzV1, TzV2, DengBO, BB1, BB3, R, 
BTT1}.
On the one hand, in the presence of such an invariant weighted Gaussian measure, 
one can study the transport property of a specific Gaussian measure, 
relying on the mutual absolute continuity 
of the invariant measure and the Gaussian measure.
On the other hand, 
the invariant measures constructed
in the forementioned work are mostly supported on rough functions
with the exception of completely integrable Hamiltonian PDEs
such as the cubic nonlinear Schr\"odinger equation (NLS), 
the KdV equation, and the Benjamin-Ono equation
\cite{Zhid, Zhid2, TzV1, TzV2, DTzV}.
These completely integrable equations admit conservation laws
at high regularities, allowing
us to construct weighted Gaussian measures supported on smooth functions.
In general, however, it is rare to have a conservation law at a high regularity
and thus one needs an alternative method to study the transport
property of Gaussian measures supported on smooth functions
under the dynamics of non-integrable PDEs.

In the following, we 
consider the cubic fourth order NLS as a model equation
and study the transport property of Gaussian measures supported
on smooth functions.
In particular, we prove that the transported Gaussian 
measures and the original Gaussian measures
are mutually absolutely continuous with respect to each other.
Our approach 
 combines 
PDE techniques such as an energy estimate and normal form reductions
and probabilistic techniques in an intricate manner.

\subsection{Cubic fourth order nonlinear Schr\"odinger equation}

As a model dispersive equation, 
we consider 
the  
cubic fourth order  nonlinear Schr\"odinger equation   on $\T$:
\begin{align}
\begin{cases}
i \dt u =   \dx^4 u \pm  |u|^{2}u \\
u|_{t = 0} = u_0,
\end{cases}
\qquad (x, t) \in \T\times \R,
\label{NLS1}
\end{align}

\noi
where  $u$ is a complex-valued function 
on $\T\times \R$ with $\T = \R/(2\pi \Z)$.
The equation \eqref{NLS1}
is also called 
the biharmonic NLS
and it was studied in \cite{IK, Turitsyn} in the context of stability of solitons in magnetic materials.
The biharmonic NLS \eqref{NLS1} 
is a special case of the following more general
class of 
fourth order NLS:
\begin{align}
i \dt u =  \ld \dx^2 u + \mu  \dx^4 u \pm  |u|^{2}u. 
\label{4NLS}
\end{align}

\noi
The model \eqref{4NLS}
was introduced in \cite{Karpman, KS}
to include the effect of 
small fourth-order dispersion terms in the propagation of intense laser beams in a bulk medium with Kerr nonlinearity. 
See also \cite{BKS, FIP, Pausader} for the references therein.


The equation \eqref{NLS1} is a Hamiltonian PDE
with the following Hamiltonian:
\begin{align}
H(u) = \frac{1}{2} \int_\T |\dx^2 u |^2 dx \pm \frac{1}{4}\int_\T |u|^4 dx. 
\label{Hamil1}
\end{align}

\noi
Moreover, the mass $M(u)$ defined by 
\begin{align}
 M(u) = \int_\T|u|^2 dx
\label{mass1}
 \end{align}

\noi
is conserved under the dynamics of \eqref{NLS1}.
This mass conservation allows us to prove 
the following global well-posedness  of \eqref{NLS1}
in $L^2(\T)$.

\begin{proposition}\label{PROP:GWP}
The cubic  fourth order  NLS \eqref{NLS1} is globally well-posed
in $H^s(\T)$
for $s \geq 0$.

\end{proposition}

\noi
See Appendix \ref{SEC:WP} for the proof.
We point out that 
Proposition \ref{PROP:GWP} is sharp
in the sense that 
 \eqref{NLS1} is ill-posed below $L^2(\T)$.
See the discussion in Subsection \ref{SUBSEC:illposed}.
See also \cite{GO, OW}.

Our main goal is to study the transport property
of Gaussian measures on Sobolev spaces
under the dynamics of \eqref{NLS1}.

\subsection{Main result}

We first introduce a family of mean-zero Gaussian measures on Sobolev spaces.
Given $ s> \frac{1}{2}$, let $\mu_s$ be the  mean-zero Gaussian measure on $L^2(\T)$
with the covariance operator $2(\text{Id} - \Dl)^{-s}$, written as
\begin{align}
 d \mu_s 
  = Z_s^{-1} e^{-\frac 12 \| u\|_{H^s}^2} du 
  = Z_s^{-1} \prod_{n \in \Z} e^{-\frac 12 \jb{n}^{2s} |\ft u_n|^2}   d\ft u_n .
\label{gauss0}
 \end{align}

\noi
While the expression  
$d \mu_s 
  = Z_s^{-1} \exp(-\frac 12 \| u\|_{H^s}^2 )du $
may suggest that $\mu_s$ is a Gaussian measure on $H^s(\T)$, 
we need to enlarge a space in order to make sense of $\mu_s$.

The Gaussian measure $\mu_s$ defined above is in fact the induced probability measure
under the map\footnote{In the following, we drop the harmless factor of $2\pi$.}
\begin{align}
 \o \in \O \mapsto u^\o(x) = u(x; \o) = \sum_{n \in \Z} \frac{g_n(\o)}{\jb{n}^s}e^{inx}, 
\label{gauss1}
 \end{align}

\noi
where  $\jb{\,\cdot\,} = (1+|\cdot|^2)^\frac{1}{2}$
and 
$\{ g_n \}_{n \in \Z}$ is a sequence of independent standard complex-valued 
Gaussian random variables, i.e.~$\text{Var}(g_n) = 2$.
Note that $u^\o$ in \eqref{gauss1} lies  in $H^{\s}(\T)$ for $\s < s -\frac 12$
but not in $H^{s-\frac 12}(\T)$ almost surely. 
Moreover, for the same range of $\s$, 
 $\mu_s$ is a Gaussian probability measure on $H^{\s}(\T)$ 
and the triplet $(H^s, H^\s, \mu_s)$ forms an abstract Wiener space.
See \cite{GROSS, Kuo2}.

Recall the following definition of quasi-invariant measures.
Given a measure space $(X, \mu)$, 
we say that $\mu$ is {\it quasi-invariant} under a transformation $T:X \to X$
if 
the transported measure $T_*\mu = \mu\circ T^{-1}$
and $\mu$
are equivalent, i.e.~mutually absolutely continuous with respect to each other.
We now state our main result.

\begin{theorem}\label{THM:quasi}
Let $s > \frac 34$.  Then, the Gaussian measure $\mu_s$ is quasi-invariant under the flow of 
the  cubic fourth order  NLS
 \eqref{NLS1}. 
\end{theorem}

When $s = 2$, one may obtain Theorem \ref{THM:quasi}
by establishing invariance of the Gibbs measure
``$d\rho = Z^{-1} \exp(-H(u))du$''
and appealing to the mutual absolute continuity of the Gibbs measure $\rho$
and the Gaussian measure $\mu_2$, 
at least in the defocusing case.
Such invariance, however, is 
a very rigid statement and is not applicable to other values of $s >\frac 34$.

Instead, 
we follow the approach 
introduced by the second author  
in the context of the (generalized) BBM equation \cite{TzBBM}.
In particular, we combine 
both PDE techniques 
and probabilistic techniques in an intricate manner.
Moreover, we perform both local and  global analysis on the phase space.
An example of local analysis 
is an energy estimate (see Proposition \ref{PROP:energy} below), 
where we study a property of a particular trajectory, 
while examples of global analysis 
include 
the transport property of Gaussian measures under global transformations discussed in Section \ref{SEC:Gauss}
and a change-of-variable formula (Proposition \ref{PROP:meas1}).

As in \cite{TzBBM}, 
it is essential to exhibit a smoothing on the nonlinear part of the dynamics of \eqref{NLS1}.
Furthermore, we crucially exploit the invariance property of the Gaussian measure
$\mu_s$ under some nonlinear (gauge) transformation.
See Section \ref{SEC:Gauss}.
In the context of the generalized BBM considered in \cite{TzBBM}, 
there was an obvious smoothing coming from the smoothing operator
applied to the nonlinearity.
There is, however, no apparent smoothing for our equation \eqref{NLS1}.
In fact, a major novelty compared to \cite{TzBBM} is that
in this work we  exploit the dispersive nature of the equation in a fundamental manner.
Our main tool in this context is  normal form reductions
analogous to the approach employed in 
\cite{BIT, KO, GKO}.
In \cite{BIT}, Babin-Ilyin-Titi introduced
a normal form approach for constructing solutions to dispersive PDEs.
It turned out that this approach has various applications
such as establishing unconditional uniqueness \cite{KO, GKO}
and exhibiting nonlinear smoothing \cite{ET}.
The normal form approach is also effective
in establishing a good energy estimate,
though such an application of the normal form reduction 
in energy estimates is more classical and precedes the work of \cite{BIT}.
See Subsection \ref{SUBSEC:energy}.

In \cite{RA}, Ramer proved a criterion on quasi-invariance
of a Gaussian measure on an abstract Wiener space
under a nonlinear transformation.  
In the context of our problem, 
this result basically states that $\mu_s$
is quasi-invariant
if the nonlinear part is $(1+\eps)$-smoother than the linear part.
See \cite{Kuo} for a related previous result. 
In Section \ref{SEC:Ramer}, 
we perform a normal form reduction on the renormalized equation \eqref{NLS5}
and exhibit $(1+\eps)$-smoothing on the nonlinear part if $s > 1$.
This argument provides the first proof of 
Theorem \ref{THM:quasi} when $ s > 1$.
It seems that the regularity restriction $s > 1$
is optimal for the application of Ramer's result.
See Remark \ref{REM:nec}.

When $s \leq 1$, we need to go beyond Ramer's argument.
In this case, we follow the basic methodology in \cite{TzBBM}, 
combining an energy estimate and global analysis of truncated measures.
Due to a lack of apparent smoothing, 
our energy estimate  is more intricate.
Indeed, we need to perform a normal form reduction and introduce a modified energy for this purpose.
This introduces a further modification to the argument from \cite{TzBBM}.
See Section \ref{SEC:quasi}.
Lastly, let us point out the following.
While
the regularity restriction $s > \frac 34$ in Theorem \ref{THM:quasi} comes from 
the energy estimate (Proposition \ref{PROP:energy}), 
we expect that, by introducing some new ideas related to more refined normal form reductions developed in \cite{GKO}, the result may be extended to the (optimal) regularity range
$s>\frac 12$.
We plan to address this question in a future work.

\begin{remark}\rm
(i)  In the higher regularity setting $s > 1$, 
we can reduce the proof of Theorem \ref{THM:quasi} to Ramer's result \cite{RA}. 
See Section~\ref{SEC:Ramer}.
While there is an explicit representation for the Radon-Nikodym derivative
in \cite{RA}, 
we do not know how to gain useful information from it at this point

\smallskip

\noi
(ii)  In the low regularity case $\frac 34 < s \leq 1$, we employ the argument introduced in \cite{TzBBM}.
See Section~\ref{SEC:quasi}.
This argument is more quantitative
and in particular, 
it allows us to obtain a polynomial upper bound on the growth of the Sobolev norm.
However, such a polynomial growth bound may also be obtained
by purely deterministic methods.
See Remark 7.4 in \cite{TzBBM}.
A quasi-invariance result with better quantitative bounds may lead
to an improvement
of the known deterministic bounds.
At the present moment, however, 
we do not know how to make such an idea work.

\smallskip

\noi
(iii) We point out that the existence of a quasi-invariant
measure is a qualitative statement, showing a delicate persistency
property of the dynamics.
In particular, this persistence property due to the quasi-invariance
is stronger than the (usual) persistence of regularity.
In a future work, 
we plan to construct Hamiltonian dynamics
possessing  the persistence of regularity
such that the Gaussian measure $\mu_s$
and the transported measure under the dynamics are mutually singular.

\end{remark}

\begin{remark}\rm
Let us briefly discuss the situation for the related cubic (second order) NLS:
\begin{align}
i \dt u =   \dx^2 u \pm  |u|^{2}u,  
\qquad (x, t) \in \T\times \R.
\label{cubicNLS}
\end{align}

\noi
It is known to be completely integrable
and possesses an infinite sequence of conservation laws $H_k$, $k \in \N \cup \{0\}$, 
controlling the $H^k$-norm
\cite{AKNS, AM, GK}.
Associated to 
the conservation laws $H_k$, $k \geq 1$, 
there exists an infinite sequence
of invariant weighted Gaussian measures $\rho_k$ 
supported on $H^{k - \frac 12 -\eps}(\T)$, $\eps > 0$
\cite{Zhid, BO94}.
As mentioned above, 
one may combine 
this invariance 
and the mutual absolute continuity 
of $\rho_k$ and the Gaussian measure $\mu_k$
to deduce quasi-invariance of $\mu_k$ under the dynamics of \eqref{cubicNLS}, $k \geq 1$.
It may be of interest to investigate quasi-invariance 
of $\mu_s$ for non-integer values of $s$.

\end{remark}

\subsection{Organization of the paper}

In Section~\ref{SEC:notations}, we introduce some notations.
In Section~\ref{SEC:2}, we apply several transformations to \eqref{NLS1}
and  derive  a new renormalized equation.
We also prove a key factorization lemma (Lemma~\ref{LEM:phase})
which play a crucial role in 
the subsequent nonlinear analysis.
We then investigate invariance properties of Gaussian measures
under several transformations in Section~\ref{SEC:Gauss}.
In Section~\ref{SEC:Ramer}, we prove  Theorem~\ref{THM:quasi}
for $s > 1$ as a consequence of Ramer's result \cite{RA}.
By establishing a crucial energy estimate and 
performing global analysis of truncated measures, we
finally present the proof of   Theorem~\ref{THM:quasi}
for the full range $s > \frac 34$ in Section~\ref{SEC:quasi}.
In Appendix~\ref{SEC:WP}, we discuss the well-posedness
issue of the Cauchy problem~\eqref{NLS1}.
Then, we use it to study the approximation property 
of truncated dynamics in Appendix~\ref{SEC:approx}, 
which is used in the proof of Theorem \ref{THM:quasi}
in Section~\ref{SEC:quasi}.

\section{Notations}
\label{SEC:notations}

Given $N \in \N$, 
we use $\P_{\leq N}$ to denote the Dirichlet projection 
onto the frequencies $\{|n|\leq N\}$
and  set $\P_{> N} := \text{Id} - \P_{\leq N}$.
Define $E_N$ and $E_N^\perp$ by 
\begin{align*}
E_N & = \P_{\leq N} L^2(\T) = \text{span}\{e^{inx}: |n|\leq N\},\\
E_N^\perp & = \P_{>N} L^2(\T) = 
\text{span}\{e^{inx}: |n|> N\}.
\end{align*}

Given $ s> \frac{1}{2}$, let $\mu_s$ be the  Gaussian measure on $L^2(\T)$
defined in \eqref{gauss0}.
Then, we can write $\mu_s$ as
\begin{align}
 \mu_s = \mu_{s, N}\otimes \mu_{s, N}^\perp,
\label{G1}
 \end{align}

\noi
where $ \mu_{s, N}$ and $\mu_{s, N}^\perp$ 
are the 
marginal distributions of $\mu_s$ restricted onto $E_N$ and $E_N^\perp$, respectively.
In other words, 
$ \mu_{s, N}$ and $\mu_{s, N}^\perp$ 
are 
induced probability measures
under the following maps: 
\begin{align}
& u_N :\o \in \O \mapsto  u_N (x; \o) = \sum_{|n|\leq N} \frac{g_n(\o)}{\jb{n}^s}e^{inx}, \label{G2}\\
& u_N^\perp: \o \in \O \mapsto u_N^\perp(x; \o) = \sum_{|n|>N} \frac{g_n(\o)}{\jb{n}^s}e^{inx}, \label{G3}
 \end{align}

\noi
 respectively.
Formally, we can write  $ \mu_{s, N}$ and $\mu_{s, N}^\perp$
as
\begin{align}
 d \mu_{s, N} = Z_{s, N}^{-1} e^{-\frac 12 \| \P_{\leq N}  u_N\|_{H^s}^2} d u_N 
\quad \text{and} \quad  
d \mu_{s, N}^\perp = \ft Z_{s,N }^{-1} e^{-\frac 12 \| \P_{>N} u_N^\perp \|_{H^s}^2} d u_N^\perp. 
\label{G4}
 \end{align}

\noi
Given $r > 0$, we also define a probability measure $\mu_{s, r}$ 
by 
\begin{align}
d \mu_{s, r} = Z_{s, r}^{-1}\ind_{\{ \| v\|_{L^2 } \leq r\}} d\mu_s.
\label{Gibbs1b}
\end{align}

The defocusing/focusing nature of the equation \eqref{NLS1} does not play any role,
and thus we assume that it is defocusing, i.e. with the $+$ sign in \eqref{NLS1}.
Moreover, in view of the time reversibility of the equation, 
we only consider positive times in the following.

\section{Reformulation of the  cubic fourth order  NLS}
\label{SEC:2}

In this section, we apply several transformations to \eqref{NLS1}
and reduce it 
to a convenient form on which we perform our analysis.
Given $t \in \R$, we define a gauge transformation $\GG_t$ on $L^2(\T)$
by setting
\begin{align}
 \mathcal{G}_t [f ]: = e^{ 2 i t \fint |f|^2} f, 
\label{gauge1}
\end{align}

\noi
where $\fint_\T f(x) dx := \frac{1}{2\pi} \int_\T f(x)dx$.
Given a function $u \in C(\R; L^2(\T))$, 
we define $\GG$ by setting
\[\GG[u](t) : = \GG_t[u(t)].\]

\noi
Note that $\GG$ is invertible
and its inverse is given by $\GG^{-1}[u](t) = \GG_{-t}[u(t)]$.

Let  $u \in C(\R; L^2(\T))$ be a solution to \eqref{NLS1}.
Define $\wt u$ by 
\begin{align}
\wt u (t) := \mathcal{G}[u](t)  = e^{ 2 i t \fint |u(t)|^2} u(t).
\label{gauge2}
\end{align}

\noi
Then, it follows from the 
 the mass conservation 
that $\wt u$ is a solution to the following renormalized fourth order NLS:
\begin{align}
i \dt \wt u  =   \dx^4 \wt u  +\bigg( |\wt u |^{2}  - 2 \fint_\T |\wt u |^2dx \bigg) \wt u. 
\label{NLS4}
\end{align}

\noi
Next, define the interaction representation $v$ of $\wt u$ by 
\begin{align}
v(t) = S(-t)\wt u(t), 
\label{gauge3}
\end{align}

\noi
where $S(t) = e^{-i t\dx^4}$.
For simplicity of notations, we use $v_n$ to denote the Fourier coefficient of $v$ in the following,
when there is no confusion.
By writing \eqref{gauge3} on the Fourier side, we have 
\begin{align}
v_n(t) = e^{ it n^4} \wt u_n(t). 
\label{gauge3a}
\end{align}

\noi
Then, with \eqref{gauge3a}, 
we can reduce
\eqref{NLS4} to the following equation for $\{v_n\}_{n \in \Z}$: 
\begin{align}
\dt v_n 
& = -i e^{i t n^4} ( i \dt \wt u_n - n^4 \wt u_n)\notag\\
& = -i \sum_{\G(n)} e^{-i \phi(\bar n) t} v_{n_1}\cj{v_{n_2}}v_{n_3}
+ i |v_n|^2 v_n \notag\\
& =: \NN(v)_n + \RR(v)_n, 
\label{NLS5}
\end{align}

\noi
where the phase function $\phi(\bar n)$ and the plane $\G(n)$ are given by
\begin{align}
 	\phi(\bar n) = \phi(n_1, n_2, n_3, n) = n_1^4 - n_2^4 + n_3^4 - n^4
\label{phi1}
\end{align}

\noi
and 
\begin{align}
\G(n) 
= \{(n_1, n_2, n_3) \in \Z^3:\, 
 n = n_1 - n_2 + n_3 \text{ and }  n_1, n_3 \ne n\}.
\label{Gam1}
 \end{align}


The phase function $\phi(\bar n)$ admits
the following factorization.
	
\begin{lemma}\label{LEM:phase}
Let $n = n_1 - n_2 + n_3$.
Then, we have
\begin{align}
\phi(\bar n) =  (n_1 - n_2)(n_1-n) 
\big( n_1^2 +n_2^2 +n_3^2 +n^2 + 2(n_1 +n_3)^2\big).
 \label{phase1}
\end{align}
	
\end{lemma}

\begin{proof}

With $n = n_1 - n_2 + n_3 $, we have	
\begin{align}
\phi(\bar n) 
& =  (n_1 - n_2)
\Big\{(n_1^3 + n_2^3
-n_3^3  - n^3 )
+ 
(  n_1^2 n_2 + n_1 n_2^2 
 - n_3^2 n  -  n_3 n^2  )\Big\} \notag \\
 & = :(n_1 - n_2)( \I + \II).
 \label{phase2}
 \end{align}

\noi
On the one hand, we have
\begin{align}
\I =  (n_1 - n) 
(  n_1^2 + n_1 n + n^2 
 + n_2^2  + n_2 n_3 +  n_3^2  ).
 \label{phase3}
\end{align}
	
\noi
On the other hand, 
with $n_2 = n_1 + n_3 - n$, we have
\begin{align}
\II & = 2n_1^3 + 3 (n_3 - n) n_1^2 
+ (n_3^2 - 2n_3 n +n^2) n_1 - n_3^2n - n_3 n^2\notag \\
& = (n_1 - n) \big(
 2n_1^2 +  (3n_3 - n) n_1 + n_3^2 + n_3 n\big).
 \label{phase4}
\end{align}
	
\noi
From \eqref{phase2} with \eqref{phase2} and \eqref{phase4}
with $n_2 = n_1 + n_3 - n$, 
we obtain
\begin{align*}
\phi(\bar n) 
& =  (n_1 - n_2)(n_1 - n)(3n_1^2 + n_2^2 + 2n_3^2 + n^2 
+ 3 n_1 n_3 + n_2 n_3 + n_3 n)\\
& =  (n_1 - n_2)(n_1 - n)\big( n_1^2 +n_2^2 +n_3^2 +n^2 + 2(n_1 +n_3)^2\big).
\qedhere
 \end{align*}

\end{proof}

In the remaining part of the paper, we present the proof of Theorem 
\ref{THM:quasi} by performing analysis on  \eqref{NLS5}.
In view of Lemma \ref{LEM:phase}, 
we refer to the first term $\NN(v)$ and the second term $\RR(v)$ 
on the right-hand side of \eqref{NLS5}
as the non-resonant and resonant terms, respectively.
While we do not have any smoothing on $\RR(v)$ under a time integration, 
  Lemma \ref{LEM:phase} shows that 
there is a smoothing on the non-resonant  term $\NN(v)$.
We will exploit this fact in Section \ref{SEC:Ramer}.
In  Section \ref{SEC:quasi}, 
we will exploit a similar non-resonant behavior
in establishing a crucial energy estimate
(Proposition \ref{PROP:energy}).

\section{Gaussian measures under transformations}
\label{SEC:Gauss}

In this section, we discuss invariance properties
of  Gaussian measures under various  transformations.

\begin{lemma}\label{LEM:gauss2}
Let $t \in \R$.  Then, the Gaussian measure $\mu_s$ defined in \eqref{gauss0} is invariant under the linear map $S(t)$.
\end{lemma}

\begin{proof}
Note that $\mu_s$ can be written as an infinite product of Gaussian measures:
\begin{align*}
\mu_s = \bigotimes_{n \in \Z} \rho_n,
\end{align*}

\noi
where $\rho_n$ is the probability distribution for $\ft u_n$.
In particular, $\rho_n$ is a mean-zero Gaussian probability measure on $\C$ with variance $2 \jb{n}^{-2s}$.
Then, noting that the action of $S(t)$ on $\ft u_n$ is a rotation by $e^{-i tn^4}$, 
the lemma follows
from the rotation invariance of each $\rho_n$.
\end{proof}

\begin{lemma}\label{LEM:gauss1}
Given  a  complex-valued mean-zero  Gaussian random variable $g$ with variance $\s$, i.e.~$g \in \NN_\C(0, \s )$, 
let $  Tg = e^{i t |g|^2} g$ for some $t \in \R$.
Then, 	$Tg  \in \NN_\C(0, \s)$. 

\end{lemma}

\begin{proof}
By viewing $\C \simeq \R^2$, 
let ${\bf x} = (x, y )  = (\Re g, \Im g)$
and ${\bf u} = 
(u, v )  = (\Re T g, \Im T g)$.
Noting that 
 $|Tg| = |g|$, we have $T^{-1}  g = e^{-it | g|^2}  g$.
In terms of ${\bf x}$ and ${\bf u}$, we have
\begin{align*}
{\bf x} = T^{-1} {\bf u}
= ( u \cos t |{\bf u}|^2 + v \sin t |{\bf u}|^2, 
-  u \sin t |{\bf u}|^2 + v \cos t |{\bf u}|^2).
\end{align*}
	
\noi	
Then, with $C_t = \cos t|{\bf u}|^2$ and $S_t = \sin t|{\bf u}|^2$, 
a direct computation yields
\begin{align*}
\det  D_{\bf u} T^{-1} 
&  = \left|\begin{matrix}
C_t  - 2t u^2 S_t  + 2t uv C_t
& 
S_t 
-2tuv S_t + 2t v^2 C_t\\
-S_t - 2t u^2 C_t - 2tuv S_t& 
C_t -  2t uv C_t - 2tv^2 S_t
\end{matrix}\right|\\
& = \big\{C_t^2 - 2 t uv C_t^2 - 2t v^2 S_t C_t\\
& \hphantom{XX}
- 2t u^2 S_t C_t + 4t^2 u^3 v S_t C_t + 4t^2 u^2 v^2 S_t^2\\
& \hphantom{XX}
+ 2 t uv C_t^2 - 4 t^2 u^2 v^2 C_t^2 - 4 t^2 u v^3 S_t C_t\big\}\\
& \hphantom{X}
 - \big\{-S_t^2 + 2t uv S_t^2 - 2t v^2 S_t C_t\\
 & \hphantom{XX}
-2 t u^2 S_t C_t + 4t^2 u^3 v S_t C_t - 4 t^2 u^2 v^2 C_t^2\\
& \hphantom{XX}
- 2t u v S_t^2 + 4 t^2 u^2 v^2 S_t^2 - 4 t^2 u v^3 S_t C_t\big\} \\
& = 1.
\end{align*}

Let $\mu$ and $\wt \mu$ be the probability distributions
for $g$ and $T g$.
Then, for a measurable set $A \subset \C \simeq \R^2 $, we have 
\begin{align*}
\wt \mu(A) & = \mu(T^{-1}A)
= \frac{1}{\pi \s}\int_{T^{-1}A} e^{-\frac{|{\bf x}|^2}{\s}} dx dy
= \frac{1}{\pi \s}\int_{A} e^{-\frac{|T^{-1} {\bf u}|^2}{\s}} |\det D_{\bf u} T^{-1}| du dv\\
& = \frac{1}{\pi\s}\int_{A} e^{-\frac{| {\bf u}|^2}{\s}}  du dv
= \mu(A). 
\end{align*}

\noi
This proves the lemma.
\end{proof}

Next, we extend Lemma \ref{LEM:gauss1} to the higher dimensional setting.

\begin{lemma}\label{LEM:gauss2a}
Let $s\in \R$ and $N \in \N$.
Then, for any $t \in \R$, 
the Gaussian measure
$\mu_{s, N}$ defined in \eqref{G4} is invariant under the map $\GG_t$
defined in \eqref{gauge1}.

\end{lemma}

While we could adapt the proof of Lemma \ref{LEM:gauss1}
to the higher dimensional setting, this would involve computing
determinants of larger and larger matrices.
Hence, 
we present an alternative proof in the following.

\begin{proof}

Given $N \in \N$, let 
$E_N = \text{span}\{ e^{inx}: |n| \leq N\}$.
Given $u \in E_N$, 
let $v(t) = \GG_t[u]$ for $t \in \R$.
Then, noting that $\dt M(v(t)) = 0$, 
where $M(v(t)) = \sum_{|n|\leq N} | v_n(t)|^2$, 
we see that $ v_n$ satisfies the following system of ODEs:
\begin{align}
d  v_n = 2 i M(v)  v_n dt, \quad |n|\leq N, 
\label{H1}
\end{align}
	
\noi
With $a_n = \Re  v_n$ and $b_n = \Im  v_n$, 
we can rewrite \eqref{H1} as
\begin{align}
\begin{cases}
d  a_n  = - 2  M(v) b_n dt\\
d  b_n  = 2  M(v) a_n dt,
\end{cases}
\quad |n| \leq N.
\label{H2}
\end{align}

Let $\L_N$ be the infinitesimal generator for \eqref{H2}.
Then, $\mu_{s, N}$ is invariant under $\GG_t$ for any $t \in \R$
if and only if $(\L_N)^*\mu_{s, N} = 0$. See \cite{KS}.
Note that the last condition is equivalent to 
\begin{align}
\int_{(a, b) \in \R^{2N+2}} \L_N F (a, b) d\mu_{s, N}(a, b) = 0
\label{H3}
\end{align}

\noi
for all test functions $F\in C^\infty(\R^{2N+2}; \R)$.
From \eqref{H2}, we have
\[ \L_N F(a, b) 
= \sum_{|n|\leq N} \bigg( - 2M(a, b) b_n \frac{\dd}{\dd a_n}
+2 M(a, b) a_n \frac{\dd}{\dd b_n}\bigg) F(a, b),\]

\noi
where $M(a, b) = \sum_{|n|\leq N} (a_n^2 + b_n^2)$.
Then, by integration by parts, we have 
\begin{align*}
\int_{(a, b) \in \R^{2N+2}}&  \L_N F (a, b) d\mu_{s, N}(a, b) \notag \\
&  = 
 2 Z_N^{-1} 
 \sum_{|n|\leq N} \int_{ \R^{2N+2}} 
F(a, b) 
\frac{\dd}{\dd a_n}\bigg\{M(a, b) b_n 
e^{-\frac12 \sum_{|k|\leq N} \frac{a_k^2}{\jb{k}^{2s}} + \frac{b_k^2}{\jb{k}^{2s}}}
\bigg\}
d a db \notag \\
& \hphantom{X}
-  2   Z_N^{-1}  
\sum_{|n|\leq N}  \int_{ \R^{2N+2}} 
F(a, b) 
\frac{\dd}{\dd b_n}\bigg\{M(a, b) a_n 
e^{-\frac12 \sum_{|k|\leq N} \frac{a_k^2}{\jb{k}^{2s}} + \frac{b_k^2}{\jb{k}^{2s}}}
\bigg\}
d a db\notag\\
 & = 
 4 Z_N^{-1} 
 \sum_{|n|\leq N} \int_{ \R^{2N+2}} 
F(a, b) 
\bigg( 1 - \frac{M(a, b)}{2\jb{n}^{2s}}\bigg)
a_n b_n e^{-\frac12 \sum_{|k|\leq N} \frac{a_k^2}{\jb{k}^{2s}} + \frac{b_k^2}{\jb{k}^{2s}}}
d a db \notag \\
& \hphantom{X}
-  4   Z_N^{-1}  
\sum_{|n|\leq N}  \int_{ \R^{2N+2}} 
F(a, b) 
\bigg( 1 - \frac{M(a, b)}{2\jb{n}^{2s}}\bigg)
a_n b_n e^{-\frac12 \sum_{|k|\leq N} \frac{a_k^2}{\jb{k}^{2s}} + \frac{b_k^2}{\jb{k}^{2s}}}
d a db \\
& = 0
\end{align*}

\noi
This proves \eqref{H3}.
\end{proof}

In the following, 
we assume that $ s> \frac 12$ such that 
$\mu_s$ is a well-defined probability measure on $L^2(\T)$
and $\GG_t$ defined in \eqref{gauge1} makes sense
on $\supp(\mu_s) =  L^2(\T)$.

\begin{lemma}\label{LEM:gauss3}
Let $s>\frac 12$.  Then, 
for any $t \in \R$, 
the Gaussian measure
$\mu_s$ defined in \eqref{gauss0} is invariant under the map $\GG_t$.
\end{lemma}

Note that, when $s = 1$, Lemma \ref{LEM:gauss3} basically follows 
from Theorem 3.1 in \cite{NRSS} which exploits the properties of the Brownian loop
under conformal mappings.
For general  $s > \frac 12$, 
such approach does not seem to be appropriate.
In the following, we present the proof, using Lemma \ref{LEM:gauss2a}.

\begin{proof}
Fix $t \in \R$.
Given $N \in \N$,
let  $F_N \in C_b(L^2(\T); \R)$
be a test function depending only on the frequencies $\{|n|\leq N\}$.
Then,  we claim that
\begin{align}
\int_{L^2} F_N\circ \GG_t (u) d \mu_s (u) = \int_{L^2} F_N  (u) d \mu_s(u). 
\label{Ginv1}
\end{align}
	
\noi
With a slight abuse of notations, 
we write
\begin{align}
F_N(u) =
F\big(\{u_n\}_{|n|\leq N}\big)
=  F_N(u_{-N}, u_{-N+1},\dots, u_{ N-1}, u_N).
\label{Ginv1a}
\end{align}

Let $v = \GG_t[u]$, 
where $u$  is as in \eqref{gauss1}.
Then, we have
\[ v_n = e^{2 i t \sum_{k \in \Z} \frac{|g_k|^2}{\jb{k}^{2s}}} \frac{g_n}{\jb{n}^s}
= e^{2 i t \sum_{|k| >N} \frac{|g_k|^2}{\jb{k}^{2s}}}
\cdot e^{2 i t \sum_{|k| \leq N} \frac{|g_k|^2}{\jb{k}^{2s}}}\frac{g_n}{\jb{n}^s}.\] 

\noi
By the independence of $\{g_n \}_{|n|\leq N}$ and $\{g_n \}_{|n|> N}$, 
we can write $\O = \O_0 \times \O_1$
such that 
\[
g_n (\o) =\begin{cases}
 g_n(\o_0), \ \o_0 \in \O_0, 
& \text{if }|n|\leq N, \\
 g_n(\o_1), \ \o_1 \in \O_1, 
& \text{if }|n|> N.
\end{cases}
\]

\noi
Then, we have 
\begin{align}
\int_{L^2} F_N\circ \GG_t (u) d \mu_s (u)
= \int_{\O_1} I_N (\o_1)
d P(\o_1), 
\label{Ginv2}
\end{align}

\noi
where $I_N(\o_1)$ is given by 
\begin{align}
I_N(\o_1)
= \int_{\O_0}
F_N\bigg( 
\bigg\{
e^{2 i t \sum_{|k| >N} 
\frac{|g_k(\o_1)|^2}{\jb{k}^{2s}}}\cdot 
e^{2 i t \sum_{|k| \leq N} 
\frac{|g_k(\o_0)|^2}{\jb{k}^{2s}}}\frac{g_n(\o_0)}{\jb{n}^s}\bigg\}_{|n|\leq N}\bigg) 
d P(\o_0).
\label{Ginv2a}
\end{align}

\noi
Since $s > \frac 12$, 
we have $\mu(\o_1) : =  \sum_{|k| >N} 
\frac{|g_k(\o_1)|^2}{\jb{k}^{2s}} < \infty $ almost surely.
For fixed $\o_1 \in \O_1$, define $\{\wt g_n^{\o_1}\}_{|n|\leq N}$ by setting 
$\wt g_n^{\o_1} = e^{2it \mu(\o_1)} g_n$, $|n| \leq N$.
Then, by the rotational invariance of the standard complex-valued Gaussian random variables
and independence of 
 $\{g_n \}_{|n|\leq N}$ and $\{g_n \}_{|n|> N}$, 
we see that, for almost every $\o_1 \in \O_1$, 
 $\{\wt g_n^{\o_1}\}_{|n|\leq N}$ 
is a sequence of independent standard complex-valued Gaussian random variables
(in $\o_0 \in \O_0$).
In particular, 
the law of  $\{\wt g_n^{\o_1}\}_{|n|\leq N}$ 
is the same as that of
 $\{ g_n\}_{|n|\leq N}$,  
almost surely in $\o_1 \in \O_1$.
Then, 
from the definitions of $\mu_{s, N}$ and $\GG_t$, 
 we can rewrite \eqref{Ginv2a} as 
\begin{align*}
I_N(\o_1)
& = \int_{\O_0}
F_N\bigg( 
\bigg\{ e^{2 i t \sum_{|k| \leq N} 
\frac{|\wt g_k^{\o_1}(\o_0)|^2}{\jb{k}^{2s}}}\frac{\wt g_n^{\o_1}(\o_0)}{\jb{n}^s}\bigg\}_{|n|\leq N}\bigg) 
d P(\o_0)\\
& = \int_{\O_0}
F_N\bigg( 
\bigg\{ e^{2 i t \sum_{|k| \leq N} 
\frac{| g_k(\o_0)|^2}{\jb{k}^{2s}}}\frac{ g_n(\o_0)}{\jb{n}^s}\bigg\}_{|n|\leq N}\bigg) 
d P(\o_0)\\
& = 
\int_{E_N}
F_N(\GG_t u_N) d\mu_{s, N}(u_N) 
\end{align*}

\noi
for almost every $\o_1 \in \O_1$,
where $u_N = \P_{\leq N} u$ is as in \eqref{G2}.	
Then, it follows from 
Lemma \ref{LEM:gauss2a}
with  \eqref{Ginv1a} and \eqref{G1} that 
\begin{align}
I_N(\o_1)
= \int_{E_N}
F_N(\GG_t u_N) d\mu_{s, N}(u_N) 
= \int_{E_N}
F_N(u_N) d\mu_{s, N}(u_N) 
= \int_{L^2}
F_N(u) d\mu_{s}(u), 
\label{Ginv3}
\end{align}

\noi
for almost every $\o_1 \in \O_1$.
Note that the right-hand side of \eqref{Ginv3} is independent of $\o_1 \in \O_1$.
Therefore,  from \eqref{Ginv2} and \eqref{Ginv3},
we have
\begin{align*}
\int_{L^2} F_N\circ \GG_t (u) d \mu_s (u)
= \int_{\O_1} 
\int_{L^2}
F_N(u) d\mu_{s}(u)
d P(\o_1)
=  \int_{L^2}
F_N(u) d\mu_{s}(u) 
\end{align*}

\noi
This proves
\eqref{Ginv1}.

Next, 
given  $F \in C_b(L^2(\T); \R)$, 
let $F_N(u) = F(\P_{\leq N} u)$, $N \in \N$.
Then, $F_N(u)$ converges to $F(u)$ almost surely with respect to $\mu_s$.
Also, 
 $F_N(\GG_t u)$ converges to $F(\GG_t u)$ almost surely with respect to $\mu_s$.
Then, from the dominated convergence theorem and \eqref{Ginv1}, we have 
\begin{align*}
\int_{L^2} F\circ \GG_t (u) d \mu_s (u) 
& = \lim_{N \to \infty} \int_{L^2} F_N\circ \GG_t (u) d \mu_s (u) 
=\lim_{N\to \infty}  \int_{L^2} F_N  (u) d \mu_s(u)\\
& = \int_{L^2} F  (u) d \mu_s(u)
\end{align*}

\noi
for all $F \in C_b(L^2(\T); \R)$.
Hence, the lemma follows
(see, for example, 
  \cite[Proposition 1.5]{DaPrato}).
\end{proof}

Lastly, we conclude this section by stating 
the invariance property of 
 quasi-invariance
under a composition of two maps.

\begin{lemma}\label{LEM:comp}
Let $(X, \mu)$ be a measure space.
Suppose that $T_1$ and $T_2$
are maps on $X$ into itself
such that $\mu$ is quasi-invariant under $T_j$ for each $j = 1, 2$.
Then, $\mu$ is quasi-invariant under  $T = T_1 \circ T_2$.

\end{lemma}
\begin{proof}

Suppose that $A \subset X$ is a measurable set 
such that $\mu(A) = 0$.
By the quasi-invariance of $\mu$ under $T_1$, 
this is equivalent to $\mu(T_1^{-1}A) = 0$.
Then, 
by the quasi-invariance of $\mu$ under $T_2$, 
the pushforward measure $T_*\mu$ satisfies
\[T_*\mu(A) = \mu(T^{-1}A) 
= \mu\big(T_2^{-1}(T_1^{-1}A)\big) = 0.\]

\noi
Conversely, if $T_*\mu(A) = 0$, 
then we have $\mu(T_1^{-1}A)=0$,
which in turn implies $\mu(A) = 0$.
Hence, $\mu$ and $T_*\mu$ are mutually absolutely continuous.
\end{proof}

\section{Ramer's argument: $s > 1$}
\label{SEC:Ramer}

In this section, we present the proof of  Theorem \ref{THM:quasi}
for $s > 1$.
Our basic approach is to apply Ramer's result
after exhibiting a sufficient smoothing
on the nonlinear part.
As it is written, 
the equation \eqref{NLS1} or \eqref{NLS5}
does not manifest a smoothing
in an explicit manner.
In the following, we perform 
a normal form reduction
and establish a nonlinear smoothing
by exploiting the dispersion of the equation.

\subsection{Normal form reduction}
By writing \eqref{NLS5} in the integral form, we have
\begin{align}
v_n (t) 
 & = v_n(0) 
 -i \int_0^t \sum_{\G(n)} 
 e^{-i \phi(\bar n) t'} v_{n_1}\cj{v_{n_2}}v_{n_3}(t') dt' 
 +i \int_0^t | v_{n}|^2v_{n}(t') dt'\notag\\
&  = : v_n(0) +\fN(v)(n, t)+\fR(v)(n, t).
\label{ZNLS1}
\end{align}

\noi
Lemma \ref{LEM:phase} states that 
we have a non-trivial (in fact, fast)
oscillation 
caused by 
the phase function $\phi(\bar n)$
in 
the non-resonant part $\fN(v)$. 
The main idea of a normal form reduction is
to transform 
the non-resonant part $\fN(v)$ into smoother terms of higher degrees, 
exploiting this rapid oscillation.
More concretely, 
integrating by parts, we formally have
\begin{align}
\fN(v)(n, t) 
& =  \sum_{\G( n)}
\frac{ e^{-i \phi (\bar n) t'} }{\phi(\bar n)}
v_{n_1}(t')\cj{v_{n_2}(t')}v_{n_3}(t')\bigg|_{t' = 0}^t 
- \sum_{\G( n)}  \int_0^t 
\frac{ e^{-i \phi(\bar n) t'} }{\phi(\bar n)}
\dt(  v_{n_1}\cj{v_{n_2}}v_{n_3})(t') dt' \notag\\
& =  
\sum_{\G(n)}
\frac{ e^{-i \phi( \bar n) t} }{\phi(\bar n)}
v_{n_1}(t)\cj{v_{n_2}(t)}v_{n_3}(t) 
- \sum_{\G( n)}
\frac{ 1}{\phi(\bar n)}
v_{n_1}(0)\cj{v_{n_2}(0)}v_{n_3}(0) \notag \\
& \hphantom{X}
-2    \int_0^t \sum_{\G(n)}
\frac{ e^{-i \phi(\bar n)t' } }{\phi(\bar n)}
\big\{ \NN(v)_{n_1} + \RR(v)_{n_1}\big\}\cj{v_{n_2}}v_{n_3}(t') dt' \notag\\
& \hphantom{X}
-   \int_0^t \sum_{\G( n)}
\frac{ e^{-i \phi(\bar n) t' } }{\phi(\bar n)}
 v_{n_1}\cj{\big\{ \NN(v)_{n_2} + \RR(v)_{n_2}\big\}}v_{n_3}(t') dt'\notag\\
&  =: \I + \II
+ \III + \IV. 
 \label{Znonlin1}
\end{align}

\noi
In view of Lemma \ref{LEM:phase}, 
the phase function $\phi(\bar n)$ appearing in the denominators
allows us to exhibit a smoothing in $\fN(v)$.
See Lemma \ref{LEM:Znonlin} below.

At this point, the computation in \eqref{Znonlin1} 
is rather formal and thus
requires justification in several steps.
In the first step, we switched the order of the time integration
and the summation:
\begin{align}
 -i \int_0^t \sum_{\G(n)} 
 e^{-i \phi(\bar n) t'} v_{n_1}\cj{v_{n_2}}v_{n_3}(t') dt' 
= 
 -i  \sum_{\G(n)} \int_0^t
 e^{-i \phi(\bar n) t'} v_{n_1}\cj{v_{n_2}}v_{n_3}(t') dt'. 
\label{just1}
\end{align}

\noi
With $w = \mathcal{F}^{-1}\big(|\ft v_n|\big) = \sum_{ n \in \Z} |\ft v_n| e^{inx}$, we have
\begin{align}
\sum_{\G(n)} 
|  v_{n_1}\cj{v_{n_2}}v_{n_3} |
\leq \|w \|_{L^3}^3 
\les \| w\|_{H^{\frac 16}}^3
= \| v\|_{H^{\frac 16}}^3.
\label{just2}
\end{align}

\noi
Hence, 
the sum $\sum_{\G(n)} 
 e^{-i \phi(\bar n) t'} v_{n_1}\cj{v_{n_2}}v_{n_3}(t') $
is absolutely convergent
with a bound uniform in time $t'$, 
provided that $v \in C(\R; H^s(\T))$ with $s \geq \frac{1}{6}$.
This justifies \eqref{just1}.

If $v \in C(\R; H^s(\T))$ with $s \geq \frac{1}{6}$, 
it follows from \eqref{NLS5}
and a computation similar to \eqref{just2} that 
$ v_n \in C^1(\R)$.
This allows us to apply integration by parts
and the product rule.
Lastly, 
we need to justify
the switching of the time integration
and the summation in the last equality of \eqref{Znonlin1}.
By crudely estimating
with \eqref{NLS5}, \eqref{just2} 
and Lemma \ref{LEM:phase} (note that $|\phi(\bar n)| \geq 1$ on $\G(n)$),  
we have
\begin{align}
 \sum_{\G(n)}
\bigg|\frac{ e^{-i \phi(\bar n)t' } }{\phi(\bar n)}
\big\{ \NN(v)_{n_1} + \RR(v)_{n_1}\big\}\cj{v_{n_2}}v_{n_3}(t')\bigg|
&\les 
\| \NN(v)_{n_1} + \RR(v)_{n_1}\|_{\l^\infty_{n_1}}
\sum_{n_2, n_3}\frac{|v_{n_2}v_{n_3}|}{\jb{n_2}\jb{n_3}} \notag\\
&\les \|v(t')\|_{H^{\frac 16}}^3 \| v(t')\|_{L^2}^2.
\label{just3}
\end{align}

\noi
Hence, 
the series on the left-hand side of \eqref{just3}
is absolutely convergent
with a bound uniform in time $t'$, 
provided that $v \in C(\R; H^s(\T))$ with $s \geq \frac{1}{6}$.
This justifies the last equality in \eqref{Znonlin1}.

The following lemma
shows a nonlinear smoothing for \eqref{NLS5}.
Note that the amount of smoothing for $\fR(v)$ depends on the regularity
$s > \frac 12$.

\begin{lemma}\label{LEM:Znonlin}
Let $s> \frac 12$.
Then, we have 
\begin{align}
\| \fN(v)(t) \|_{H^{s+2}} & \les
\|v(0)\|_{H^s}^3 + \|v (t)\|_{H^s}^3 + t 
\sup_{t' \in [0, t]} \|v (t')\|_{H^s}^5, 
\label{Znonlin2}\\
\| \fR(v)(t) \|_{H^{3s}} & \les
  t  \sup_{t' \in [0, t]}  \|v (t')\|_{H^s}^3.
 \label{Znonlin3}
\end{align}
	
\end{lemma}

\begin{proof}
By Lemma \ref{LEM:phase} and the algebra property of $H^s(\T)$, $s > \frac 12$, we have
\begin{align*}
\| \I \|_{H^{s+2}} 
& \les \bigg\| \jb{n}^{s} \sum_{\G(n)}
|v_{n_1}(t)\cj{v_{n_2}(t)}v_{n_3}(t)| \bigg\|_{\l^2_n}
\les \bigg\| \jb{n}^{s} \sum_{n = n_1 - n_2 + n_3}
\prod_{j = 1}^3 |v_{n_j}(t)| \bigg\|_{\l^2_n} \\
& = \big\| \{\F^{-1}(|\ft v_n|)(t)\}^3\big\|_{H^s}
\les  \| v(t)\|_{H^s}^3.
\end{align*}

\noi
The second term $\II$ in \eqref{Znonlin1} can be estimated in an analogous manner.
Similarly, by Lemma \ref{LEM:phase},  \eqref{NLS5}, 
and the algebra property of $H^s(\T)$, $s > \frac 12$, we have
\begin{align*}
\| \III \|_{H^{s+2}} 
& \les 
t \sup_{t'\in [0, t]}
\bigg\| \jb{n}^s \sum_{\G(n)}
\big|\big\{ \NN(v)_{n_1} + \RR(v)_{n_1}\big\}\cj{v_{n_2}}v_{n_3}(t')\big|\bigg\|_{\l^2_n}\\
& = 
t \sup_{t'\in [0, t]}
\big\| \{\F^{-1}(|\ft v_n|)(t')\}^5\big\|_{H^s}
\les  t \sup_{t'\in [0, t]} \| v(t')\|_{H^s}^5.
\end{align*}

\noi
The fourth term $\IV$ in \eqref{Znonlin1} can be estimated in an analogous manner.
From \eqref{ZNLS1} and 
$\l^2_n \subset \l^6_n$, we have
\begin{align*}
\|\fR(v)(t)\|_{H^{3s}}
& \leq t \sup_{t'\in [0, t]}
\bigg(\sum_{n \in \Z} \jb{n}^{6s} | v_{n}(t')|^6\bigg)^\frac{1}{2}
= t \sup_{t'\in [0, t]}
\| \jb{n}^s v_n(t')\|_{\l^6_n}^3\\
& \leq t \sup_{t'\in [0, t]}
\| \jb{n}^s v_n(t')\|_{\l^2_n}^3
=  t \sup_{t'\in [0, t]}
\| v(t')\|_{H^s}^3.
\end{align*}

\noi
This proves the 
second estimate \eqref{Znonlin3}.
\end{proof}

\subsection{Consequence of Ramer's result}

In this subsection, we present the proof of  Theorem \ref{THM:quasi}
for  $s > 1$.
The main ingredient is Ramer's result \cite{RA}
along with the nonlinear smoothing discussed in the previous 
subsection.
We first recall the precise statement 
of the main result in  \cite{RA}
for readers' convenience.

\begin{proposition}[Ramer
\cite{RA}]\label{PROP:RA0}
Let $(i, H, E)$ be an abstract Wiener space and 
$\mu$ be the standard Gaussian measure on $E$.
Suppose that $T = \textup{Id} + K: U \to E$
be a continuous (nonlinear) transformation
from some open subset   $U\subset E$ into $E$
such that
\begin{itemize}
\item[(i)] $T$ is a homeomorphism of $U$ onto an open subset of $E$.
\item[(ii)]
We have $K(U) \subset H$
and $K:U \to H$ is continuous.

\item[(iii)]
For each $x \in U$, 
the map $DK(x)$ is a Hilbert-Schmidt operator on $H$.
Moreover, $DK: x \in U \to DK(x) \in  HS(H)$ is continuous.

\item[(iv)]
$\textup{Id}_H + DK(x) \in GL(H)$ for each $x \in U$.

\end{itemize}

\noi

\noi
Then, $\mu$ and $\mu\circ T $
are mutually absolutely continuous measures on $U$.
\end{proposition}

\noi
Here,  $HS(H)$ denotes
the space of Hilbert-Schmidt operators on $H$
and $GL (H)$ denotes invertible linear operators on $H$
with a bounded inverse.

Given $t, \tau \in \R$, let 
$\Phi(t):L^2\to L^2$ be the solution map for \eqref{NLS1}
and 
$ \Psi(t, \tau):L^2\to L^2$ be the solution map for \eqref{NLS5},\footnote{Note that \eqref{NLS5}
is non-autonomous.
We point out that this non-autonomy does not play an essential
role in the remaining part of the paper, since all the estimates hold uniformly in $t \in \R$.
}
sending initial data at time $\tau$
to solutions at time $t$. 
When $\tau =0$, we may denote $\Psi(t, 0)$ by $\Psi(t)$ for simplicity.

By inverting the transformations  \eqref{gauge2} and  \eqref{gauge3} 
with  \eqref{ZNLS1}, 
we have
\begin{align}
\Phi(t)(u_0) = \GG^{-1} \circ S(t) \circ \Psi(t)(u_0)
=  \GG^{-1} \circ  S(t) (u_0 + \fN(v)(t)+\fR(v)(t)), 
\label{nonlin2}
\end{align}

\noi
where $v(t) = \Psi(t)(u_0)$ and $\fN$ is given by \eqref{Znonlin1}.
Now, write  $\Psi(t) = \text{Id} + K(t)$, where
\[ K(t)(u_0) : =  \fN(v)(t)+\fR(v)(t) \]

\noi
and $v$ is the solution to \eqref{NLS5} with $v|_{t = 0} = u_0$.
In view of Lemmas \ref{LEM:gauss2}, \ref{LEM:gauss3}, and \ref{LEM:comp}, 
it suffices to show that $\mu_s$ is quasi-invariant under
$\Psi(t)$.

Fix $s > 1$ and $\s_1 > \frac{1}{2}$ sufficiently close to $\frac 12$.
First, note that $\mu_s$ is a probability measure on $H^{s-\s_1}(\T)$.
Given $R>0$, let $B_R$ be the open ball of radius $R$ centered at the origin in 
$H^{s - \s_1}(\T)$.
The following proposition
shows that the hypotheses of Ramer's result in \cite{RA}
are indeed satisfied.

\begin{proposition}\label{PROP:RA}
Let $s > 1$.  Given $R > 0$, 
there exists $\tau = \tau(R) > 0$ such that,
for each $t \in  (0, \tau(R)]$, the following statements hold:

\begin{itemize}
\item[(i)] 

$\Psi(t)$ is a homeomorphism of $B_R$
onto an open subset of $H^{s-\s_1}(\T)$.

\smallskip

\item[(ii)]  
We have $K(t) (B_R) \subset H^s(\T)$ and $K(t): B_R \to H^s(\T)$ is continuous.

\smallskip

\item[(iii)] 
For each $u_0 \in B_R$, 
the map $DK(t)|_{u_0}$ is a Hilbert-Schmidt operator on $H^s(\T)$.
Moreover, 
$DK(t): u_0 \in B_R \mapsto DK(t)|_{u_0} \in HS(H^s(\T))$
is continuous.

\smallskip

\item[(iv)] $\textup{Id}_{H^s} + DK(t)|_{u_0} \in GL (H^s(\T))$ 
for each $u_0 \in B_R$.

\end{itemize}

\end{proposition}

We first present the proof of 
Theorem \ref{THM:quasi} for $s > 1$,
assuming Proposition \ref{PROP:RA}.
Thanks to Ramer's result (Proposition \ref{PROP:RA0} above), 
 Proposition \ref{PROP:RA}
implies that
 $\mu_s$ and 
 the pullback measure $\Psi(t)^*\mu_s := \mu_s\circ\Psi(t)$
are mutually absolutely continuous
as measures restricted to the ball $B_R$
 for any $t \in (0, \tau(R)]$.

\begin{proof}[Proof of Theorem \ref{THM:quasi} for $s > 1$]
Given $R>0$, 
let $B_R \subset 
H^{s - \s_1}$ be the open ball of radius $R$ centered at the origin as above.
Fix $T > 0$. 
It follows from the growth estimate \eqref{LWP4} of  the $H^{s-\s_1} $-norm 
that 
\begin{align}
\sup_{t \in [0, T]} \|v(t)\|_{H^{s-\s_1}} \leq C(T, R)  =: R^*
\label{quasi1}
\end{align}

\noi
for all solutions $v$ to \eqref{NLS5} with $v|_{t = 0} \in B_R$.

Suppose that $A \in \mathcal{B}_{H^{s-\s_1}}$
is a Borel set in $H^{s-\s_1}$ such that 
$\mu_s(A) = 0$.
Given $R>0$, let $R^*$ be as in \eqref{quasi1}.
Then, from \eqref{quasi1}, we have 
\begin{align}
\Psi(t) (A\cap B_R ) 
\subset B_{R^*}
\label{quasi2}
\end{align}

\noi
for all $t \in [0, T]$.
Noting that 
$ \mu_s(A\cap B_R) =0$, 
it follows from 
Proposition \ref{PROP:RA}
and the result in \cite{RA}
that 
\begin{align}
 \mu_s(\Psi(t)(A\cap B_R)) =0
\label{quasi3}
 \end{align}

\noi
for any $0 \leq t \leq \tau$, 
where  $\tau = \tau(R^*)$  is as in Proposition \ref{PROP:RA}.
In view of  \eqref{quasi2}, 
we can iteratively apply
Proposition \ref{PROP:RA}
and the main result in \cite{RA} on time interval $[j \tau, (j+1) \tau]$
and see that \eqref{quasi3} holds for all $t \in [0, T]$.
In particular, we have 
\[ \mu_s(\Psi(T)(A\cap B_R)) =0.\]

\noi
Now, letting $R \to \infty$, 
it follows from the continuity from below of a measure  that 
$ \mu_s(\Psi(T)(A))=0$.
Note that the choice of $T$ was arbitrary.
In view of  the time reversibility of the equation \eqref{NLS5}, 
we conclude that 
 $\mu_s$ is quasi-invariant under the flow of \eqref{NLS5}.
Therefore,  Theorem \ref{THM:quasi} follows from 
\eqref{nonlin2} with
 Lemmas \ref{LEM:gauss2}, \ref{LEM:gauss3}, and \ref{LEM:comp}.
\end{proof}

The remaining part of this section is devoted to the proof of Proposition \ref{PROP:RA}.
The  claim (i) follows from the well-posedness of \eqref{NLS5} in $H^{s-\s_1}$.
In particular, the continuity of $\Psi(t)$ on $H^{s-\s_1}$ with the time reversibility 
implies (i).
As before, from \eqref{LWP4}, we have the uniform growth bound:
\begin{align}
\sup_{t \in [0, \tau]} \|v(t)\|_{H^{s-\s_1}} \leq C^{R^\theta \tau} R =: R_\tau
\label{Q1}
\end{align}

\noi
for all solutions $v$ to \eqref{NLS5} with $v|_{t = 0} = u_0 \in B_R$.
Then, 
the  claim 
(ii) follows from  Lemma \ref{LEM:Znonlin}
and the continuity of $\Psi(t)$ on $H^{s-\s_1}$.

We postpone the proof of the claim (iii)
and first prove the claim (iv).
For fixed $u_0 \in B_R \subset H^{s-\s_1}$ and $t \in \R$, 
define a map $F:H^s \to H^s$
by 
\[ F(h) = \Psi(t)(u_0 + h) - u_0=  h + K(t) (u_0 + h) , \quad  h \in H^s.\]

\noi
Then, by computing a derivative of $F$ at the origin, we have\footnote{
By viewing $( H^s, H^{s-\s_1}, \mu_s)$ as an abstract Wiener space, 
$DF|_{0}$
 is  the so-called $H$-derivative of  $F$ at $0$,
 where $H = H^s$ is the Cameron-Martin space.
 See \cite{RA}.}
$DF|_{0} = \textup{Id}_{H^s} + DK(t)|_{u_0}$.
This is clearly a linear map.
Moreover, the boundedness of $DF|_{0}$ on $H^s$
follows from the claim (iii).
Note that $F$ is invertible
with the inverse $F^{-1}$ given by 
\[ F^{-1}(h) =
\Psi(-t)(u_0 + h) - u_0 
=  h + K(-t) (u_0 + h).\]

\noi
Hence, it follows from the chain rule
that $DF|_{0} = \textup{Id}_{H^s} + DK(t)|_{u_0}$ is invertible.
Moreover, we have
\begin{equation}
 (DF|_{0})^{-1} =  \textup{Id}_{H^s} + DK(-t)|_{u_0}.
\label{Q2}
 \end{equation}

\noi
Hence, we proved the claim (iv) except for
the boundedness of $(DF|_{0})^{-1}$.
We will prove the boundedness of $(DF|_{0})^{-1}$
 at the end of this section.

Next, we prove the claim (iii).
In the following, we will prove that 
$DK(t)|_{u_0}$ is Hilbert-Schmidt on $H^s$
for 
$u_0 \in B_R \subset H^{s-\s_1}$
as long as $t = t(R) \ll1$. 
Given $u_0 \in B_R \subset H^{s-\s_1}$, 
let $v$ be the global solution to \eqref{NLS5}
with $v|_{t = 0} = u_0$.

We first introduce some notations.
Given a multilinear\footnote{By multilinearity, we mean it is either linear or conjugate linear in each argument, 
i.e.~linear over real numbers.} expression $\mathcal{M}(v, v, \dots, v)$, 
we use $\mathcal{M}(v^*, v^*, \dots, v^*)$
to denote the sum of the form
$\mathcal{M}(w, v, \dots) + \mathcal{M}(v, w,v,  \dots)$, 
where each  multilinear  term has exactly one factor of $w$ and the remaining arguments are $v$.
For example, 
we have 
\begin{align*}
v^*_{n_1}(t)\cj{v^*_{n_2}(t)}v^*_{n_3}(t) 
= w_{n_1}(t)\cj{v_{n_2}(t)}v_{n_3}(t) 
+ v_{n_1}(t)\cj{w_{n_2}(t)}v_{n_3}(t) 
+v_{n_1}(t)\cj{v_{n_2}(t)}w_{n_3}(t). 
\end{align*}

\noi
We use a similar convention for multilinear expressions in $v(0)$.
In this case, we use 
 $\mathcal{M}(v^*(0), v^*(0), \dots)$
to denote the sum of the form
$\mathcal{M}(w(0), v(0), \dots) + \mathcal{M}(v(0), w(0),v(0),  \dots)$, 
where each multilinear term has exactly one factor of $w(0)$ and the remaining arguments are $v(0)$.

Let $w(t)$ be a solution to the following linear equation:
\begin{align}
\begin{cases}
\displaystyle \dt w_n = -i \sum_{\G(n)} e^{-i \phi(\bar n) t} v^*_{n_1}\cj{v^*_{n_2}}v^*_{n_3}
+ i |v^*_n|^2 v^*_n\\
w|_{t= 0} = w(0).
\end{cases}
\label{Zlin2}
\end{align}

\noi
Given $(m_1, m_2, m_3) \in \Z^3$
and $n \in \Z$,  we use the following shorthand notation:
\begin{align}
 (\bar m, n) := (m_1, m_2, m_3, n).
 \label{Gam2}
 \end{align}

\noi
Then, 
by a direct computation with 
\eqref{NLS5}, 
\eqref{ZNLS1},  and \eqref{Znonlin1},  we have
\begin{align*}
\F\big[DK(t)|_{u_0}(w(0))\big](n)
& =  i \int_0^t |v^*_n(t')|^2v^*_{n}(t') dt'
\notag\\
& \hphantom{X}
+ 
\sum_{\G( n)}
\frac{ e^{-i \phi( \bar n) t} }{\phi(\bar n)}
v^*_{n_1}(t)\cj{v^*_{n_2}(t)}v^*_{n_3}(t) 
- \sum_{\G( n)}
\frac{ 1}{\phi(\bar n)}
v^*_{n_1}(0)\cj{v^*_{n_2}(0)}v^*_{n_3}(0) \notag \\
& \hphantom{X}
+2 i   \int_0^t 
\sum_{\substack{
(n_1, n_2, n_3) \in \G(n)\\
(m_1, m_2, m_3) \in \G(n_1)}}
\frac{ e^{-i \phi(\bar n)t' -i \phi(\bar m, n_1)t'} }{\phi(\bar n)}
 v^*_{m_1}\cj{ v^*_{m_2}} v^*_{m_3}\cj{v^*_{n_2}}v^*_{n_3}(t') dt' \notag\\
& \hphantom{X}
-2 i   \int_0^t \sum_{\G( n)}
\frac{ e^{-i \phi(\bar n)t' } }{\phi(\bar n)}
 |v^*_{n_1}|^2v^*_{n_1}\cj{v^*_{n_2}}v^*_{n_3}(t') dt' \notag\\
& \hphantom{X}
-  i \int_0^t 
\sum_{\substack{
(n_1, n_2, n_3) \in \G(n)\\
(m_1, m_2, m_3) \in \G(n_2)}}
\frac{ e^{-i \phi(\bar n) t' +i \phi(\bar m, n_2) t'} }{\phi(\bar n)}
 v^*_{n_1}\cj{v^*_{m_1}}v^*_{m_2}\cj{v^*_{m_3}}v^*_{n_3}(t') dt' \notag\\
& \hphantom{X}
+  i \int_0^t \sum_{\G( n)}
\frac{ e^{-i \phi(\bar n) t' } }{\phi(\bar n)}
 v^*_{n_1}|v^*_{n_2}|^2\cj{v^*_{n_2}}v^*_{n_3}(t') dt', 
\end{align*}
	
\noi
where 
$\phi(\bar n)$ and $\G(n)$ are as in \eqref{phi1} and \eqref{Gam1}, respectively.

Fix $\s_2 > \frac 12$ (to be chosen later)
and 
write
\begin{align*}
DK(t)|_{u_0}(w(0)) = \jb{\dx}^{-\s_2} \circ A_t(w(0))
\end{align*}

\noi
where $A_t(w(0))$ is given by
\begin{align}
\F\big[A_t(w(0))\big](n)
& =  i \int_0^t \jb{n}^{\s_2} |v^*_n(t')|^2v^*_{n}(t') dt'
\notag\\
& \hphantom{X}
+ 
\sum_{\G( n)}
\frac{ e^{-i \phi( \bar n) t} }{\phi(\bar n)}
\jb{n}^{\s_2} v^*_{n_1}(t)\cj{v^*_{n_2}(t)}v^*_{n_3}(t) 
- \sum_{\G(n)}
\frac{ 1}{\phi(\bar n)}
\jb{n}^{\s_2} v^*_{n_1}(0)\cj{v^*_{n_2}(0)}v^*_{n_3}(0) \notag \\
& \hphantom{X}
+2 i   \int_0^t 
\sum_{\substack{
(n_1, n_2, n_3) \in \G(n)\\
(m_1, m_2, m_3) \in \G(n_1)}}
\frac{ e^{-i \phi(\bar n)t' -i \phi(\bar m, n_1)t'} }{\phi(\bar n)}
\jb{n}^{\s_2}  v^*_{m_1}\cj{ v^*_{m_2}} v^*_{m_3}\cj{v^*_{n_2}}v^*_{n_3}(t') dt' \notag\\
& \hphantom{X}
-2 i   \int_0^t \sum_{\G( n)}
\frac{ e^{-i \phi(\bar n)t' } }{\phi(\bar n)}
\jb{n}^{\s_2}  |v^*_{n_1}|^2 v^*_{n_1}\cj{v^*_{n_2}}v^*_{n_3}(t') dt' \notag\\
& \hphantom{X}
-  i \int_0^t 
\sum_{\substack{
(n_1, n_2, n_3) \in \G(n)\\
(m_1, m_2, m_3) \in \G(n_2)}}
\frac{ e^{-i \phi(\bar n) t' +i \phi(\bar m, n_2) t'} }{\phi(\bar n)}
\jb{n}^{\s_2}  v^*_{n_1}\cj{v^*_{m_1}}v^*_{m_2}\cj{v^*_{m_3}}v^*_{n_3}(t') dt' \notag\\
& \hphantom{X}
+  i \int_0^t \sum_{\G( n)}
\frac{ e^{-i \phi(\bar n) t' } }{\phi(\bar n)}
\jb{n}^{\s_2}  v^*_{n_1}|v^*_{n_2}|^2\cj{v^*_{n_2}}v^*_{n_3}(t') dt', 
\label{Zlin3}
\end{align}

\noi
Note that $\jb{\dx}^{-\s_2}$ is a Hilbert-Schmidt operator on $H^s$.
Thus, 
if we  prove that $A_t$ is bounded on $H^s$,
then it follows that 
$DK(t)|_{u_0}$ is Hilbert-Schmidt on $H^s$.
Hence, we focus on proving the boundedness of $A_t$ on $H^s$
in the following.

Let $t \in [0,  1]$.
Given $ s > 1$, choose $\s_1, \s_2 > \frac{1}{2}$ such that 
\begin{align}
s- \s_1 > \tfrac{1}{2}, 
\quad 
\tfrac{s+\s_2}{3} \leq s - \s_1, 
\quad \text{and}
\quad
s + \s_2 - 2 \leq s - \s_1.
\label{Zlin3a}
\end{align}

\noi
Applying Young's inequality and $\l^2_n \subset \l^6_n$
to \eqref{Zlin3} with Lemma \ref{LEM:phase}
and \eqref{Zlin3a}, we have
\begin{align}
\|A_t(w(0))\|_{H^s} 
\les  \sup_{t' \in [0, t]}
\bigg\{ & \|v(t')\|^2_{H^\frac{s+\s_2}{3}}\|w(t')\|_{H^\frac{s+\s_2}{3}}
 + \|v(t')\|^2_{H^{s-\s_1}} \|w(t')\|_{H^{s-\s_1}}\notag\\
& 
+ \|v(t')\|^4_{H^{s-\s_1}} \|w(t')\|_{H^{s-\s_1}}
\bigg\}.
\label{Zlin4}
\end{align}

Given $\tau > 0$, 
it follows from \eqref{LWP4} that
\begin{align}
\sup_{t \in [0, \tau]}\|v(t)\|_{H^{s - \s_1 }}
\leq C(R).
\label{Zlin5}
\end{align}


\noi
Then, 
from \eqref{Zlin2} and \eqref{Zlin5} with \eqref{Zlin3a}, we have
\begin{align*}
\sup_{t \in [0, \tau]}\|w(t)\|_{H^{s - \s_1 }}
& \leq \|w(0)\|_{H^{s}} + C \tau \sup_{t \in [0, \tau]}\| v(t)\|^2_{H^{s - \s_1 }}\| w(t)\|_{H^{s - \s_1 }}\notag \\
& \leq \|w(0)\|_{H^{s}} + C(R) \tau  \sup_{t \in [0, \tau]}\| w(t)\|_{H^{s - \s_1 }}.
\end{align*}

\noi
In particular, by 
choosing  $\tau = \tau(R) > 0$ sufficiently small, 
we obtain 
\begin{align}
\sup_{t \in [0, \tau]}\|w(t')\|_{H^{s - \s_1 }} 
\les \|w(0)\|_{H^{s}}.
\label{Zlin6}
\end{align}

\noi
Finally, it follows from \eqref{Zlin4}, \eqref{Zlin5}, and \eqref{Zlin6}
with \eqref{Zlin3a} that
\begin{align*}
\|A_t(w(0))\|_{H^s} 
\leq C(R) \|w(0)\|_{H^{s}}.
\end{align*}

\noi
Therefore, $A_t$ is bounded on $H^s$ 
and hence 
$DK(t)|_{u_0}$ is a Hilbert-Schmidt operator on $H^s$
for all $t \in [0, \tau]$.
The second claim in (iii) basically follows from 
the continuous dependence of \eqref{NLS5} and \eqref{Zlin2} (in $v$)
and thus we omit details.

It remains to prove  the boundedness of $(DF|_{0})^{-1} = (\textup{Id}_{H^s} + DK(t)|_{u_0})^{-1}$
By  the time reversibility of the equation
and 
 \eqref{Q2}, 
 the argument above shows that $(DF|_0)^{-1} - \text{Id}_{H^s}$
 is Hilbert-Schmidt on $H^s$
by choosing $\tau = \tau(R)$ sufficiently small.
In particular,  $(DF|_0)^{-1}$
 is bounded on $H^s$.
This completes the proof of Proposition \ref{PROP:RA}.

 \begin{remark}\label{REM:nec}\rm
 The condition $s > 1$ 
 is necessary for this argument.
 In estimating 
 the resonant term, i.e.~the first term in  
 \eqref{Zlin3} by the $H^{s-\s_1}$-norms of its arguments, 
 we need to use
 the second condition $\tfrac{s+\s_2}{3} \leq s - \s_1$ in \eqref{Zlin3a}.
Thus, we must have 
\[s \geq \frac{3\s_1 + \s_2}{2} > 1,\]

\noi
since $\s_1, \s_2 > \frac 12$.
 
 \end{remark}

\section{Proof of Theorem \ref{THM:quasi}: $s > \frac 34$}
\label{SEC:quasi}

In this section, we present the proof of Theorem \ref{THM:quasi}
for $s > \frac 34$.
The basic structure of our argument 
follows 
the argument introduced in \cite{TzBBM} by the second author 
in the context of the (generalized) BBM equation, 
with one importance difference.
While the energy estimate in \cite{TzBBM}
was carried out on the $H^s$-norm of
solutions (to the truncated equations),
we carry out our energy estimate
on a modified energy.
This introduction of a modified energy
is necessary to exhibit a hidden nonlinear smoothing,
exploiting the dispersion of the equation.
See Proposition \ref{PROP:energy} below.
 This, in turn, forces us to work
 with the weighted Gaussian measure
 $\rho_{s, N, r, t}$ and $\rho_{s, r, t}$ 
 adapted to this modified energy, 
instead of the Gaussian measure $\mu_{s, r}$
with an $L^2$-cutoff.  
See \eqref{Gibbs1} and \eqref{Gibbs1a} below for the definitions
of  $\rho_{s, N, r, t}$ and $\rho_{s, r, t}$. 
Lastly, we point out that this usage of the modified energy is 
close to the spirit of higher order modified energies in the $I$-method
introduced by Colliander-Keel-Staffilani-Takaoka-Tao \cite{CKSTT1, CKSTT2}.

%


%
%

As in Section \ref{SEC:Ramer}, 
we carry out our analysis on \eqref{NLS5}.
Let us first introduce the following truncated approximation to \eqref{NLS5}:
\begin{align}
\dt v_n 
& = 
\ind_{|n|\leq N}\bigg\{-i 
\sum_{\G_N(n)} e^{-i \phi(\bar n) t} v_{n_1}\cj{v_{n_2}}v_{n_3}
+ i |v_n|^2 v_n\bigg\},  
\label{NLS6}
\end{align}

\noi
where $\G_N(n)$ is defined by 
\begin{align}
\G_N(n) & = \G(n) \cap \{ (n_1, n_2, n_3) \in \Z^3: |n_j|\leq N\}\notag\\
& = \{(n_1, n_2, n_3) \in \Z^3:\, 
 n = n_1 - n_2 + n_3,  \,  n_1, n_3, \ne n, 
 \text{ and } |n_j| \leq N\}.
\label{Gam3}
\end{align}

\noi
A major part of this section is devoted to 
the study of the dynamical properties of \eqref{NLS6}.
Note that \eqref{NLS6} is an infinite dimensional system ODEs
for the Fourier coefficients $\{ v_n \}_{n \in \Z}$, 
where the flow is constant on the high frequencies $\{|n|> N\}$.

We also consider the following finite dimensional system of ODEs:
\begin{align}
\dt v_n = 
-i 
\sum_{\G_N(n)} e^{-i \phi(\bar n) t} v_{n_1}\cj{v_{n_2}}v_{n_3}
+ i |v_n|^2 v_n,  \qquad |n| \leq N.
\label{NLS7}
\end{align}

%

\noi
Given $t, \tau \in \R$, 
 denote by $\Psi_N(t, \tau)$ and $\wt \Psi_N(t, \tau)$
the solution maps of \eqref{NLS6}
and \eqref{NLS7}, 
sending initial data at time $\tau$
to solutions at time $t$, respectively.
For simplicity, we set 
\begin{align}\Psi_N(t) = \Psi_N(t, 0)
\qquad \text{and}
\qquad 
\wt \Psi_N(t) = \wt \Psi_N(t, 0)
\label{Psit}
\end{align}

\noi
 when $\tau = 0$. 
Then, we have the following relations:
\begin{align}
 \Psi_N(t, \tau) = \wt \Psi_N (t, \tau) \P_{\leq N} + \P_{>N}
\qquad \text{and}
\qquad
\P_{\leq N} \Psi_N(t, \tau) = \wt \Psi_N (t, \tau) \P_{\leq N}.
\label{flow1}
 \end{align}

\subsection{Energy estimate}\label{SUBSEC:energy}

In this subsection, we establish a key energy estimate.
Before stating the main proposition, let us first perform
a preliminary computation.
Given   a smooth solution $u$ to \eqref{NLS1}, 
let $v$ be as in \eqref{gauge3}.
Then, from \eqref{NLS5}, we have
\begin{align}
 \frac{d}{dt} \| u (t) \|_{H^s}^2
=  \frac{d}{dt} \|v(t) \|_{H^s}^2
= - 2\Re i 
\sum_{n \in \Z} \sum_{\G(n)}
e^{ - i \phi(\bar n) t} \jb{n}^{2s} v_{n_1} \cj{v_{n_2}} v_{n_3} \cj{v_n}.
\label{E0}
\end{align}

\noi
Then, differentiating by parts, 
i.e.~integrating by parts without an integral sign,\footnote{This is indeed  a normal form reduction 
 applied to the evolution equation \eqref{E0} for $\|v(t)\|_{H^s}^2$.
 Compare this with the normal form reduction argument in Section \ref{SEC:Ramer}.}
we obtain
\begin{align}
\frac{d}{dt} \|v(t) \|_{H^s}^2
& = 
 2\Re \frac{d}{dt} \bigg[
\sum_{n \in \Z} \sum_{\G(n)}
\frac{e^{ - i \phi(\bar n) t}}{ \phi(\bar n)} \jb{n}^{2s} v_{n_1} \cj{v_{n_2}} v_{n_3} \cj{v_n}\bigg]\notag \\
& \hphantom{X}
- 2\Re 
\sum_{n \in \Z} \sum_{\G(n)}
\frac{e^{ - i \phi(\bar n) t}}{ \phi(\bar n)} \jb{n}^{2s} 
\dt (v_{n_1} \cj{v_{n_2}} v_{n_3} \cj{v_n}).
\label{E1}
\end{align}

\noi
This motivates us to define the following quantity.
Given $s > \frac 12$, define the modified energy $E_t(v)$ by
\begin{align}
E_t(v) & = \|v\|_{H^s}^2 
- 2\Re 
\sum_{n \in \Z} \sum_{\G(n)}
\frac{e^{ - i \phi(\bar n) t}}{ \phi(\bar n)} \jb{n}^{2s} v_{n_1} \cj{v_{n_2}} v_{n_3} \cj{v_n} \notag\\
& =: \|v\|_{H^s}^2  + R_t(v).
\label{Hamil2}
\end{align}

\noi
Then, we have the following energy estimate.

\begin{proposition}\label{PROP:energy}
Let $s > \frac 34$.
Then, 
for any sufficiently small $\eps > 0$, 
there exist small $\theta  > 0$
and $C>0$
such that 
\begin{align} 
\bigg|\frac{d}{dt} E_t(\P_{\leq N} v) \bigg|
\leq C
 \| v \|_{L^2}^{4+ \theta} \|v\|_{H^{s - \frac 12 - \eps}}^{2-\theta}, 
\label{E1a}
\end{align}

\noi
for all $N \in \N$
and any solution $v$ to \eqref{NLS6},
 uniformly in $t \in \R$.

\end{proposition}

Recall that the probability  measures $\mu_s$ and $\mu_{s, r}$
defined in \eqref{gauss0} and \eqref{Gibbs1b}
are supported on $H^{s - \frac 12 -\eps}(\T)$ for any $\eps > 0$,
while we have $\|v \|_{L^2} \leq r$ in the support of $\mu_{s, r}$.	
	
Before proceeding to the proof of this proposition, 	
recall the following arithmetic fact \cite{HW}.
Given $n \in \N$, the number $d(n)$ of the divisors of $n$
satisfies
\begin{align}
d(n) \leq C_\dl n^\dl
\label{divisor}
\end{align}

\noi
for any $\dl > 0$.

\begin{proof}

Let $v$ be a solution to \eqref{NLS6}.
Then, from \eqref{E1} and  \eqref{Hamil2} with \eqref{NLS6}, we have
\begin{align}
\frac d{dt} E_t(\P_{\leq N} v)
& = 
\NN_1(v)+\RR_1(v)
+ \NN_2(v)+\RR_2(v)
+\NN_3(v)+\RR_3(v), 
\label{E2-}
\end{align}

\noi
where  $\NN_j(v)$ and $\RR_j(v)$, $j = 1, 2, 3$, are defined by 
\begin{align}
\NN_1(v)(t) 
&:= 
 4\Re  i   \sum_{|n|\leq N} 
\sum_{\substack{
(n_1, n_2, n_3) \in \G_N(n)\\
(m_1, m_2, m_3) \in \G_N(n_1)}}
\frac{ e^{-i \phi(\bar n)t -i \phi(\bar m, n_1)t} }{\phi(\bar n)}
\jb{n}^{2s}
 v_{m_1}\cj{ v_{m_2}} v_{m_3}\cj{v_{n_2}}v_{n_3} \cj{v_n} \notag\\
\RR_1(v)(t) 
&:= 
- 4 \Re i   \sum_{|n|\leq N} \sum_{\G_N( n)}
\frac{ e^{-i \phi(\bar n)t } }{\phi(\bar n)}
\jb{n}^{2s} |v_{n_1}|^2v_{n_1}\cj{v_{n_2}}v_{n_3}\cj{v_n} \notag\\
\NN_2(v)(t) 
&:= 
 - 2\Re i
\sum_{|n|\leq N}
\sum_{\substack{
(n_1, n_2, n_3) \in \G_N(n)\\
(m_1, m_2, m_3) \in \G_N(n_2)}}
\frac{ e^{-i \phi(\bar n) t +i \phi(\bar m, n_2) t} }{\phi(\bar n)}
\jb{n}^{2s} v_{n_1}\cj{v_{m_1}}v_{m_2}\cj{v_{m_3}}v_{n_3}\cj{v_n} \notag\\
\RR_2(v)(t) 
&:= 
 2\Re  i \sum_{|n|\leq N} \sum_{\G_N( n)}
\frac{ e^{-i \phi(\bar n) t } }{\phi(\bar n)}
\jb{n}^{2s} v_{n_1}|v_{n_2}|^2\cj{v_{n_2}}v_{n_3}\cj{v_n} \notag\\
\NN_3(v)(t) 
&:= 
-  2\Re i 
\sum_{|n|\leq N} 
\sum_{\substack{(n_1, n_2, n_3) \in \G_N( n)\\(m_1, m_2, m_3) \in \G_N( n)}}
\frac{ e^{-i \phi(\bar n) t +i \phi(\bar m, n) t} }{\phi(\bar n)}
\jb{n}^{2s} v_{n_1} \cj{v_{n_2}}v_{n_3}
 \cj{v_{m_1}}v_{m_2}\cj{v_{m_3}} \notag\\
\RR_3(v)(t) 
&:= 
2\Re  i \sum_{|n|\leq N} \sum_{\G_N( n)}
\frac{ e^{-i \phi(\bar n) t } }{\phi(\bar n)}
\jb{n}^{2s} v_{n_1}\cj{v_{n_2}}v_{n_3}|v_{n}|^2\cj{v_n}.
\label{E2}
\end{align}

\noi
Here, $(\bar m, n) = (m_1, m_2, m_3, n)$
and $(\bar m, n_j) = (m_1, m_2, m_3, n_j)$ are 
as in \eqref{Gam2}.
For simplicity of the presentation, 
we drop the restriction on the summations in \eqref{E2}
with the understanding that $v_n = 0$ for $|n|> N$. 
Moreover, we can assume that all the Fourier coefficients are non-negative.
In the following, we establish 
uniform (in $t$) estimates for these multilinear terms
$\NN_j$ and $\RR_j$, $j = 1, 2, 3$. 
For simplicity, we suppress the $t$-dependence
with the understanding that 
all the estimates hold with implicit constants independent of $t \in \R$.

Given $n, \mu \in \Z$, define $\G(n, \mu)$ by 
\begin{align*}
 \G(n, \mu) 
: \! & =\G(n) \cap \{ (n_1, n_2, n_3) \in \Z^3:  \mu = (n-n_1)(n-n_3)\} \notag \\
& = \{ (n_1, n_2, n_3) \in \Z^3: n = n_1 - n_2 + n_3, \ \notag\\
& \hphantom{XXXXXXXXXX||}
n_1, n_3 \ne n,\   \mu = (n-n_1)(n-n_3)\}. 
\end{align*}

\noi
Then, 
given $\dl > 0$, 
it follows from the divisor counting estimate
\eqref{divisor} that
\begin{align}
\# \G(n, \mu)=  \sum_{\G(n, \mu)} 1 \leq C_\dl |\mu|^\dl.
\label{Ediv}
\end{align}

\noi
In the following, 
we use \eqref{Ediv} to estimate $\NN_j(v)$ and $\RR_j(v)$, $j = 1, 2, 3$.
For simplicity of the presentation, we drop multiplicative constants depending on $\dl > 0$.

We now estimate $\NN_1(v)$.
We first consider the case $s <1 $.
By Sobolev's inequality and interpolation, we have 
\begin{align}
\bigg\|\sum_{\G(n_1)} v_{m_1}\cj{ v_{m_2}} v_{m_3}\bigg\|_{\l^\infty_{n_1}}
& \leq \|\F^{-1}(|\ft v_n|)\|_{L^3}^3 \les \|v\|_{H^\frac{1}{6}}^3 
 \les \| v \|_{L^2}^{1+\theta} \|v\|_{H^{\frac{1}{4}+\g}}^{2-\theta}
\label{E2a}
\end{align}

\noi
for small $\g > 0$ and some $\theta = \theta(\g) >0$.
Then, 
by Lemma \ref{LEM:phase}
and Cauchy-Schwarz inequality (in $n$ and then in $ n_2, n_3$) with \eqref{E2a}, we have 
\begin{align}
|\NN_1(v)|
& \les
   \sum_{n \in \Z} \sum_{\mu \ne 0} 
   \sum_{\G(n, \mu)}
\frac{1}{|\mu|n_{\max}^{2-2s}}
v_{n_2}v_{n_3} 
v_n
\bigg\|\sum_{\G(n_1)} v_{m_1}\cj{ v_{m_2}} v_{m_3}\bigg\|_{\l^\infty_{n_1}} \notag\\
& \leq
\| v \|_{L^2}^{2+\theta} \|v\|_{H^{\frac{1}{4}+\g}}^{2-\theta}\notag\\
& \hphantom{XX}
\times
 \Bigg\{   \sum_{n \in \Z} 
\bigg(\sum_{\mu \ne 0} 
   \frac{1}{|\mu|^{1+2\dl}}
   \sum_{\G( n, \mu)} 1\bigg)
\bigg(\sum_{\G(n)}
\frac{v_{n_2}^2 v_{n_3} ^2 }{|(n_2 - n_3)(n-n_3)|^{1-2\dl}n_{\max}^{4-4s}}\bigg)\Bigg\}^{\frac 12}
\label{E3}
\end{align}

\noi
for small $\g, \dl > 0$
such that $ 5 -  4s - 2\dl  > 1$,
where $n_{\max} : = \max(|n|, |n_1|, |n_2|, |n_3|)$.
From the divisor counting argument \eqref{Ediv}, 
we have 
\begin{align}
|\NN_1(v)|
& \les
\| v \|_{L^2}^{2+\theta} \|v\|_{H^{\frac{1}{4}+\g}}^{2-\theta}
 \Bigg\{   \sum_{n \in \Z} 
\bigg(\sum_{\mu \ne 0} 
   \frac{1}{\mu^{1+2\dl}} |\mu|^{\dl}\bigg)
\bigg(\sum_{\G(n)}
\frac{
v_{n_2}^2 v_{n_3} ^2 }{|(n - n_1)(n-n_3)|^{1-2\dl}n_{\max}^{4-4s}}\bigg)\Bigg\}^{\frac 12}  \notag\\
& \les
\| v \|_{L^2}^{2+\theta} \|v\|_{H^{\frac{1}{4}+\g}}^{2-\theta}
 \Bigg\{  
 \sum_{n_2, n_3 \in \Z} v_{n_2}^2 v_{n_3} ^2 
  \sum_{n \ne n_3} 
\frac{1}{|n-n_3|^{1-2\dl}\jb{n}^{4-4s}}\Bigg\}^{\frac 12}  \notag\\
& \les
\| v \|_{L^2}^{4+\theta} \|v\|_{H^{s - \frac 12 - \eps}}^{2-\theta}
\label{E4}
\end{align}

\noi
for sufficiently small $\dl, \eps, \g > 0$, 
provided that  $s > \frac 34$.

Next, we consider the case $s \geq 1$.
Note that for $(n_1, n_2, n_3) \in \G(n)$, 
we have $\max_{j = 1,2, 3}|n_j|\ges |n|$.
First, suppose that $|n_1| \ges |n|$.
In this case, we use the fact that 
 $\max_{j = 1,2, 3}|m_j|\ges |n_1|$
for $(m_1, m_2, m_3) \in \G (n_1)$.
Without loss of generality, assume that  $|m_1|\ges |n_1| (\ges |n|)$.
Proceeding as in \eqref{E3} and \eqref{E4} with \eqref{E2a}, 
Cauchy-Schwarz inequality, 
\eqref{Ediv}, 
and interpolation, we have
\begin{align}
|\NN_1(v)|
& \les
   \sum_{n \in \Z} \sum_{\mu \ne 0} 
   \sum_{\G(n, \mu)}
\frac{1}{|\mu| \jb{n}^{2\dl}}
{v_{n_2}}v_{n_3} 
\jb{n}^{s - 1+\dl}  {v_n} 
\bigg\|\sum_{\G(n_1)} \jb{m_1}^{s - 1+\dl} v_{m_1}\cj{ v_{m_2}} v_{m_3}\bigg\|_{\l^\infty_{n_1}} \notag\\
& \les
\| v\|_{H^{s-1+\dl}}
\| v\|_{H^{s-\frac56 +\dl }}
\| v\|_{H^{\frac 16}}^2
\notag\\
& \hphantom{XX}
\times
 \Bigg\{   \sum_{n \in \Z} 
\bigg(\sum_{\mu \ne 0} 
   \frac{1}{|\mu|^{1+2\dl}}
      \sum_{\G( n, \mu)}1   \bigg)
\bigg(\sum_{\G(n)}
\frac{v_{n_2}^2 v_{n_3} ^2 }{|(n - n_1)(n-n_3)|^{1-2\dl}\jb{n}^{4\dl}}\bigg)\Bigg\}^{\frac 12}\notag\\
& \les
\| v\|_{H^{s-1+ \dl}}
\| v\|_{H^{s-\frac56 +\dl }}
\| v\|_{H^{\frac 16}}^2
\notag\\
& \hphantom{XX}
\times
 \Bigg\{   \sum_{n \in \Z} 
\bigg(\sum_{\mu \ne 0} 
   \frac{1}{|\mu|^{1+2\dl}}
|\mu|^\dl  \bigg)
\bigg(\sum_{\G(n)}
\frac{v_{n_2}^2 v_{n_3} ^2 }{|(n - n_1)(n-n_3)|^{1-2\dl}\jb{n}^{4\dl}}\bigg)\Bigg\}^{\frac 12}\notag\\
& \les
\| v\|_{H^{s-1+ \dl}}
\| v\|_{H^{s-\frac56 +\dl}}
\| v\|_{H^{\frac 16}}^2
 \Bigg\{   \sum_{n_2, n_3 \in \Z}
 v_{n_2}^2 v_{n_3} ^2
 \sum_{n \ne n_3}
\frac{1 }{|n-n_3|^{1-2\dl}\jb{n}^{4\dl}}\Bigg\}^{\frac 12}\notag\\
& \les 
\|v\|_{L^2}^2
\| v\|_{H^{s-1 + \dl}}
\| v\|_{H^{s-\frac56 +\dl}}
\| v\|_{H^{\frac 16}}^2
 \les 
\| v \|_{L^2}^{4+\theta} \|v\|_{H^{s - \frac 12 - \eps}}^{2-\theta}
\label{E5}
\end{align}

\noi
for sufficiently small $\dl, \eps >0$ and some $\theta = \theta (s,\dl,  \eps) > 0$.

Suppose that $|n_1|\ll |n|$.
In this case, we have $\max(|n_2|, |n_3|) \ges |n|$.
Without loss of generality, assume that $|n_2 |\ges |n|$.
Proceeding as in \eqref{E3} and \eqref{E4} with \eqref{Ediv} and \eqref{E2a}, we have 
\begin{align}
|\NN_1(v)|
& \les
   \sum_{n \in \Z} \sum_{\mu \ne 0} 
   \sum_{\G(n, \mu)}
\frac{1}{|\mu|\jb{n}^{2\dl}}
\jb{n_2}^{s-1+\dl}{v_{n_2}}v_{n_3} 
\jb{n}^{s - 1+\dl}  {v_n} 
\bigg\|\sum_{\G(n_1)}  v_{m_1}\cj{ v_{m_2}} v_{m_3}\bigg\|_{\l^\infty_{n_1}} \notag\\
& \les \|v\|_{H^{s-1+\dl}}\|v\|_{H^\frac{1}{6}}^3
 \Bigg\{   \sum_{n \in \Z} 
\bigg(\sum_{\mu \ne 0} 
   \frac{1}{|\mu|^{1+2\dl}}
|\mu|^\dl  \bigg)
\bigg(\sum_{\G(n)}
\frac{\jb{n_2}^{2(s-1+\dl)}v_{n_2}^2 v_{n_3} ^2 }{|(n - n_1)(n-n_3)|^{1-2\dl}\jb{n}^{4\dl}}\bigg)\Bigg\}^{\frac 12}\notag\\
& \les \|v\|_{H^{s-1+\dl}}\|v\|_{H^\frac{1}{6}}^3
 \Bigg\{   \sum_{n_2, n_3 \in \Z}
\jb{n_2}^{2(s-1+\dl)} v_{n_2}^2 v_{n_3} ^2
 \sum_{n \ne n_3}
\frac{1 }{|n-n_3|^{1-2\dl}\jb{n}^{4\dl}}\Bigg\}^{\frac 12}\notag\\
& \les
\| v\|_{L^2}
\| v\|_{H^{s-1+\dl}}^2
\| v\|_{H^{\frac 16}}^3
 \les \| v\|_{L^2}^{4+\theta}
\| v\|_{H^{s-\frac 12 -\eps}}^{2-\theta}
\label{E6}
\end{align}

\noi
for sufficiently small $\dl, \eps >0$ and some $\theta = \theta (s, \dl, \eps) > 0$.

Noting that $\mu = (n-n_1) (n - n_3) = (n_2 - n_1) (n_2 - n_3)$
under $n = n_1 - n_2 + n_3$, 
we can estimate
$\NN_2(v)$ and $\NN_3(v)$  in a similar
manner.

Next,  we estimate $\RR_1(v)$.
The remaining terms $\RR_2(v)$ and $\RR_3(v)$ can be estimated in a similar manner.
Without loss of generality, suppose that $|n_3|\ges |n|$.
From Lemma \ref{LEM:phase} and the divisor counting argument \eqref{Ediv}, we have
\begin{align}
|\RR_1(v)|
& \les
 \|v\|_{H^{s-1+\dl}}
 \Bigg\{   \sum_{n \in \Z} 
\bigg(\sum_{\mu \ne 0} 
   \frac{1}{\mu^{1+2\dl}} |\mu|^{\dl}\bigg)
\bigg(\sum_{\G(n)}
\frac{
v_{n_1}^6
v_{n_2}^2 \jb{n_3}^{2(s-1+\dl)} v_{n_3} ^2 }{|(n - n_1)(n-n_3)|^{1-2\dl}\jb{n}^{4\dl}}\bigg)\Bigg\}^{\frac 12}  \notag\\
& \les
\| v \|_{L^2}^{4} \|v\|_{H^{s - 1+\dl}}^{2}
\les \| v \|_{L^2}^{4+ \theta} \|v\|_{H^{s - \frac 12 - \eps}}^{2-\theta}
\label{E7}
\end{align}

\noi
\noi
for sufficiently small $\dl, \eps >0$ and some $\theta = \theta (s, \dl, \eps) > 0$.

Therefore, \eqref{E1a} follows from \eqref{E4}, 
 \eqref{E5},  \eqref{E6}, and \eqref{E7}.
This completes the proof of Proposition \ref{PROP:energy}.
\end{proof}

%
%
%
%
%

\subsection{Weighted Gaussian measures}\label{SUBSEC:meas}

Our main goal in this subsection is to define 
weighted Gaussian  measures 
adapted to the modified energy $E_t(\P_{\leq N} v)$
and $E_t(v)$ defined in the previous section.
Given $N \in \N$ and $r>0$, 
define $F_{N, r, t}(v)$ and $F_{r, t}(v)$ by 
\begin{align} F_{N, r, t}(v) = \ind_{\{ \| v\|_{L^2 } \leq r\}} e^{-\frac 12 R_t(\P_{\leq N} v)}
\quad \text{and}\quad  
F_{r, t}(v) = \ind_{\{ \| v\|_{L^2 } \leq r\}} e^{-\frac 12 R_t( v)},
\label{Gibbs0}
\end{align}

\noi
where $R_t$ is defined in \eqref{Hamil2}.
Then, we would like to construct 
probability measures 
$\rho_{s, N, r, t}$
and $\rho_{s, r, t}$ of the form:\footnote{
The normalizing constants
$Z_{s, N, r}$ and $Z_{s, r}$
a priori depend on $t \in \R$.
It is, however, easy to see that 
they are indeed independent of $t \in \R$
by (i) noticing that $R_t(v)$ defined in 
\eqref{Hamil2}
is autonomous in terms of $\wt u (t)= S(t) v(t)$
and 
(ii)  the invariance of $\mu_s$ under $S(t)$ (Lemma \ref{LEM:gauss2}).}
\begin{align}
d \rho_{s, N, r, t}
&  = \text{``}Z_{s, N, r}^{-1} \ind_{\{ \| v\|_{L^2 } \leq r\}} e^{- \frac 12 E_t(\P_{\leq N} v)} dv\text{''} \notag\\
& =  
Z_{s, N, r}^{-1} F_{N, r, t} d\mu_s
\label{Gibbs1}
\end{align}
	
\noi
and 
\begin{align}
d \rho_{s, r, t}
&  = \text{``}Z_{s, r}^{-1} \ind_{\{ \| v\|_{L^2 } \leq r\}} e^{- \frac 12 E_t( v)} dv\text{''} \notag\\
& =  
  Z_{s, r}^{-1} F_{r, t} d\mu_s.
\label{Gibbs1a}
\end{align}

\noi
The following proposition shows that they are indeed
well defined probability measures on $H^{s-\frac 12 -\eps}(\T)$, $\eps > 0$.

\begin{proposition}\label{PROP:Gibbs}
Let $s > \frac 12$ and $r > 0$.
Then, $F_{N, r, t}(v) \in L^p(\mu_s)$ for any $p \geq 1$
with a uniform bound in $N \in \N$ and $t \in \R$, depending only on $p\geq 1$ and $r > 0$.
Moreover, for any finite $p \geq 1$, 
$F_{N, r, t}(v)$ converges to  $F_{r, t}(v)$ in $L^p (\mu_s)$,
uniformly in $t\in \R$,  as $N \to \infty$.

\end{proposition}
	
\noi	
In the following, we restrict our attention to $s > \frac 12$.
Hence, we view $\rho_{s, N, r, t}$
and $\rho_{s, r, t}$ as probability measures on $L^2(\T)$.

Let $\mu_{s, r}$ be as in \eqref{Gibbs1b}. 
Then, it follows from Proposition \ref{PROP:Gibbs}
that $\rho_{s, r, t}$ and $\mu_{s, r}$
are mutually absolutely continuous.
Moreover, we have the following `uniform convergence' property of $\rho_{s, N, r, t}$
to $\rho_{s, r, t}$.

\begin{corollary}\label{COR:Gibbs}
Given $s > \frac 12$ and  $r>0$,  let $\rho_{s, N, r, t}$
and $\rho_{s, r, t}$ be as in \eqref{Gibbs1}
and \eqref{Gibbs1a}.
Then, for any  $\g > 0$, there exists $N_0 \in \N$ such that 
\[ |  \rho_{s, N, r, t}(A) - \rho_{s, r, t}(A)| < \g \]

\noi
for any $N \geq N_0$
and  any measurable set 
$A \subset L^2(\T)$,
uniformly in $t\in \R$.
\end{corollary}

The proof of Proposition \ref{PROP:Gibbs} follows
closely 
Bourgain's argument in constructing Gibbs measures \cite{BO94}.
We first recall the following basic tail estimate.
See \cite[Lemma 4.2]{OQV} for a short proof.

\begin{lemma} \label{LEM:polar}
Let $\{ g_n\}_{n \in \N}$ be independent standard complex-valued Gaussian random variables.
Then, there exist constant $c, C>0$ such that,  
 for any $M\geq 1$, we have the following tail estimate:
\begin{equation*} 
P\bigg[ \Big( \sum_{n=1}^M |g_n|^2\Big)^\frac{1}{2} \geq K \bigg]
\leq  e^{-cK^2}, \qquad K   \geq C M^\frac{1}{2} .
\end{equation*}
\end{lemma}

\begin{proof}[Proof of Proposition \ref{PROP:Gibbs}]
Fix $r> 0$.
We first prove 
\begin{align}
\| F_{N, r, t} \|_{L^p(\mu_s)}, \, \| F_{r, t}\|_{L^p(\mu_s)} \leq C_{p, r} < \infty
\label{Gibbs2}
\end{align}
	
\noi
for all $N \in \N$
and $t \in \R$.
From the distributional characterization of the $L^p$-norm
and \eqref{Gibbs0}, we have
\begin{align*}
 \| F_{r, t} \|_{L^p(\mu_s)}^p
& = p \int_0^\infty \ld^{p-1} \mu_s(|F_{ r, t}| > \ld) d\ld\\
& \leq  C+ p \int_e^\infty \ld^{p-1} 
\mu_s \big(|R_t(v)| \geq \log \ld, \, \|v\|_{L^2} \leq r\big) 
d\ld.
\end{align*}

\noi
In the following, 
we estimate
$\mu_s \big(|R_t(v)| \geq K, \, \|v\|_{L^2} \leq r\big) $
for $K \geq 1$,
using the dyadic pigeon hole principle and Lemma \ref{LEM:polar}.
Let us divide the argument into two cases: $ s> 1$ and $\frac 12 < s \leq 1$.
Note that, while $R_t$ depends on $t \in \R$, 
all the estimates below hold uniformly in $t\in \R$.

First, suppose that $s > 1$.
Then, from Lemma \ref{LEM:phase} and the divisor counting argument
as in the proof of Proposition \ref{PROP:energy}
(see \eqref{E7}), 
we have
\begin{align}
|R_t( v)| \leq C_0  \|  v\|_{L^2}^2
 \| v\|_{H^{s-1}}^2
\leq  C_0  r^2
 \| v\|_{H^{s-1}}^2.
\label{Gibbs3}
\end{align}

\noi
under $\| v\|_{L^2 } \leq r$.
Similarly, we have 
\begin{align}
|R_t(\P_{\leq M_0} v)| \leq C_0 M_0^{2(s-1)} \| \P_{\leq M_0} v\|_{L^2}^4
\leq C_0 M_0^{2(s-1)}r^4. 
\label{Gibbs4}
\end{align}
	
\noi
Given $K \geq 1$, choose  $M_0 > 0$ such that 
\begin{align}
\tfrac12 K = C_0 M_0^{2(s-1)}r^4. 
\label{Gibbs5}
\end{align}

\noi
For $j \in \N$, let $M_j = 2^j M_0$
and 
$\s_j = C_\eps 2^{-\eps j}$ for some small $\eps > 0$
such that  $\sum_{j \in \N } \s_j = \frac 12$.
Then, 
from \eqref{Gibbs3} and \eqref{Gibbs4}, 
we have
\begin{align*}
\mu_s \big(|R_t(v)| \geq K, \, \|v\|_{L^2} \leq r\big) 
& \leq \mu_s \big(\|v\|_{H^{s-1}}^2  \geq C_0^{-1}r^{-2} K\big) \notag \\
& \leq \sum_{j = 1}^\infty 
\mu_s \big(\|\P_{M_j} v\|_{H^{s-1}}^2  \geq \s_j C_0^{-1}r^{-2} K\big) \notag\\
& \les \sum_{j = 1}^\infty 
P\bigg[ \Big( \sum_{|n|\sim M_j }|g_n|^2\Big)^\frac{1}{2}
  \ges L_j \bigg], 
\end{align*}
	
\noi
where $L_j := (\s_j r^{-2} K)^\frac{1}{2}   M_j
\ges M_0^{\frac 12 \eps} M_j^{1-\frac 12 \eps} \gg M_j^{\frac 12}$.
Here, we used that $r^{-2} K \sim M_0^{s-1} r^2 \ges1$ in view of \eqref{Gibbs5}.
Then, applying Lemma \ref{LEM:polar} with \eqref{Gibbs5}, we obtain 
\begin{align*}
\mu_s \big(|R_t(v)| \geq K, \, \|v\|_{L^2} \leq r\big) 
&  \les \sum_{j = 1}^\infty 
e^{-c L_j^2}
=  \sum_{j = 1}^\infty 
e^{-c_r 2^{(2 - \eps)j}  M_0^{2} K } \notag\\
& =  \sum_{j = 1}^\infty 
e^{-c'_r 2^{(2 - \eps)j}  K^{1+\frac{1}{s-1} }}
\les 
e^{-c''_r   K^{1+\frac{1}{s-1} }}.
\end{align*}

\noi
This proves \eqref{Gibbs2} for $F_{r, t}$ when $s > 1$.
A similar argument holds for $F_{N, r, t}$
with a uniform bound in $N \in \N$.

Next, suppose that  $\frac 12<  s \leq 1$. 
Proceeding with 
Lemma \ref{LEM:phase}
as before, 
we have
\begin{align}
|R_t( v)| \les 
 \| v\|_{H^{\frac s2 -\frac 12}}^4 \leq r^4
\label{Gibbs5a}
\end{align}

\noi
under $\| v\|_{L^2 } \leq r$.
Hence, 
\eqref{Gibbs2} trivially follows in this case.

It remains to show that $F_{N, r, t}$ converges to $F_{r, t}$ in $L^p(\mu_s)$.
It  follows
from a small modification of \eqref{Gibbs3}
and \eqref{Gibbs5a} that
$R_t(\P_{\leq N} v)$
converges to 
$R_t( v)$ 
almost surely with respect to $\mu_s$, 
uniformly in $t\in \R$.
Indeed, when $s > 1$, we have
\begin{align*}
|R_t( v) - R_t(\P_{\leq N} v)| 
 \les \| \P_{> N} v\|_{L^2}
  \|  v\|_{L^2}
 \| v\|_{H^{s-1}}^2
+ 
 \|  v\|_{L^2}^2
\| \P_{> N} v\|_{H^{s-1}}
 \| v\|_{H^{s-1}} \longrightarrow 0, 
\end{align*}

\noi
while we have
\begin{align*}
|R_t( v) - R_t(\P_{\leq N} v)| \les 
 \| \P_{> N} v\|_{H^{\frac s2 -\frac 12}}
 \| v\|_{H^{\frac s2 -\frac 12}}^3
\longrightarrow 0, 
\end{align*}

\noi
when $ \frac 12  < s\leq 1$.
Hence, $F_{N, r, t}$ converges to $F_{r, t}$
almost surely with respect to $\mu_s$. 
As a consequence of Egoroff's theorem, 
we see that  $F_{N, r, t}$ converges to $F_{r, t}$ almost uniformly and hence in measure
(uniformly in $t \in \R$).
Namely, given $\eps > 0$, 
if we let 
\[ A_{N, \eps, t} = \{ v \in L^2(\T):\, |F_{N, r, t}(v) - F_{r, t}(v)|\leq \tfrac 12 \eps\}, \]

\noi
we have 
$\mu_s(A_{N, \eps, t}^c) \to 0$, 
uniformly in $t \in \R$, as $N \to \infty$.
Then, by Cauchy-Schwarz inequality and \eqref{Gibbs2}, we have  
\begin{align*}
\| F_{N, r, t} - F_{r, t}\|_{L^p(\mu_s)}
& \leq \| (F_{N, r, t} - F_{r, t})\ind_{A_{N, \eps, t}}\|_{L^p(\mu_s)} 
+\| (F_{N, r, t} - F_{r, t})\ind_{A_{N, \eps, t}^c}\|_{L^p(\mu_s)} \\
& \leq \tfrac 12 \eps 
+\| F_{N, r, t} - F_{r, t}\|_{L^{2p}(\mu_s)}
\|\ind_{A_{N, \eps, t}^c}\|_{L^{2p}(\mu_s)} \\
& \leq \tfrac{1}{2}\eps 
+ 2C_{2p, r} 
\mu_s(A_{N, \eps, t}^c)^\frac{1}{2p} \\
& \leq \eps
\end{align*}

\noi
for all sufficiently large $N \in \N$,
uniformly in $t \in \R$.
Therefore, $F_{N, r, t}$ converges to $F_{r, t}$ in $L^p(\mu_s)$ for any $p \geq 1$.
\end{proof}

We conclude this subsection by stating a large deviation estimate
on the quantity appearing in 
the energy estimate (Proposition \ref{PROP:energy}).

\begin{lemma}\label{LEM:LD}
Let $\eps > 0$ and $r>0$.
Then, there exists $C = C(\eps, r) > 0$ such that 
\begin{align*}
\big\| \| v \|_{H^{s-\frac 12 - \eps}} \big\|_{L^p(\rho_{s, N, r, t})}\leq C p^\frac{1}{2}
\end{align*}

\noi
for any $p \geq 2$, any $t \in \R$, and all sufficiently large $N \in \N$.
\end{lemma}
	
\begin{proof}

By Proposition \ref{PROP:Gibbs}, we have 
\begin{align*}
\big\| \| v   \|_{H^{s-\frac 12 - \eps}} \big\|_{L^p(\rho_{s, N, r, t})}
&  \leq \| F_{N, r}\|_{L^{2p}(\mu_s)}
\big\| \| v \|_{H^{s-\frac 12 - \eps}} \big\|_{L^{2p}(\mu_s)}\notag \\
& \les \bigg\| \Big\|\sum_{n \in \Z} \frac{g_n}{\jb{n}^{\frac 12 + \eps}}e^{inx} \Big\|_{L^2_x} \bigg\|_{L^{2p}(\O)}\notag\\
&  \les p^\frac{1}{2}
   \bigg\| \Big\| \sum_{n \in \Z} \frac{g_n}{\jb{n}^{\frac 12 + \eps}}e^{inx}  \Big\|_{L^{2}(\O)}\bigg\|_{L^2_x}
\les p^\frac {1}{2}.
\end{align*}

\noi
Here, the second to the last inequality follows
from the hypercontractvity estimate due to Nelson \cite[Theorem 2]{Nelson}.
See also  \cite[Lemma 3.1]{BT1}.
\end{proof}

\subsection{A change-of-variable formula}

In this subsection, we establish an important change-of-variable formula
(Proposition \ref{PROP:meas1}).
It is 
strongly motivated by the work \cite{TzV1, TzV2}.
We  closely follow the argument presented in \cite{TzBBM}.

Given $N \in \N$, let $d L_N = \prod_{|n|\leq N} d\ft u_n $ denote the Lebesgue
measure on $\C^{2N+1}$.
Then, 
from \eqref{Gibbs0}
and \eqref{Gibbs1} with \eqref{G4}, we have 
\begin{align*}
d \rho_{s, N, r, t} 
& = Z_{s, N, r}^{-1} 
\ind_{\{ \| v\|_{L^2 } \leq r\}} e^{-\frac 12 R_t(\P_{\leq N} v)}
 d\mu_s\notag\\
& =\ft Z_{s, N, r}^{-1} 
\ind_{\{ \| v\|_{L^2 } \leq r\}} e^{-\frac 12 E_t(\P_{\leq N} v)}
dL_N \otimes
 d\mu_{s, N}^\perp, 
\end{align*}

\noi
where $\ft Z_{s, N, r}$ is a normalizing constant defined by\footnote{
The normalizing constant $\ft Z_{s, N, r}$
a priori depends on $t \in \R$.
Arguing as for 
$Z_{s, N, r}$ and $Z_{s, r}$
defined in  \eqref{Gibbs1} and \eqref{Gibbs1a}, 
however, we see that 
it is indeed independent of $t \in \R$.}
\begin{align*}
\ft Z_{s, N, r}  = \int_{L^2} \ind_{\{ \| v\|_{L^2 } \leq r\}} 
e^{-\frac 12 E_t( \P_{\leq N}  v)} d L_N \otimes d\mu_{s, N}^\perp.
\end{align*}

\noi
Then, we have the following change-of-variable formula:

\begin{proposition}\label{PROP:meas1}
Let $s > \frac 12$, $N \in \N$, and $r > 0$.
Then, we have 
\begin{align}
\rho_{s, N, r, t}(\Psi_N(t, \tau)(A))
& = Z_{s, N, r}^{-1}\int_{\Psi_N(t, \tau)(A)} \ind_{\{ \| v\|_{L^2 } \leq r\}} e^{-\frac 12 R_t( \P_{\leq N} v)} d\mu_s(v)
\notag \\
& = \ft Z_{s, N, r}^{-1}\int_{A} \ind_{\{ \| v\|_{L^2 } \leq r\}} 
e^{-\frac 12 E_t( \P_{\leq N} \Psi_N(t, \tau) (v))} d L_N \otimes d\mu_{s, N}^\perp
\label{meas1}
\end{align}

\noi	
for any $t, \tau \in \R$ and any measurable set 
$A \subset L^2$.


\end{proposition}

We first state the basic invariance property of $L_N$.

\begin{lemma}\label{LEM:meas1a}
Let $N \in \N$.
Then,  the Lebesgue measure 
$d L_N = \prod_{|n|\leq N} d\ft u_n $ 
is invariant under the flow $\wt \Psi_N(t, \tau)$.

\end{lemma}

\begin{proof}
The finite dimensional system \eqref{NLS7}
basically corresponds to the finite dimensional Hamiltonian approximation
to \eqref{NLS1} under two transformations
 \eqref{gauge2} and \eqref{gauge3}.
Therefore, morally speaking,
 the lemma should follow from 
the inherited Hamiltonian structure and 
Liouville's theorem.
In the following, however, we provide a direct proof.

Write  \eqref{NLS7} as $\dt v_n = X_n$, $|n| \leq N$.
Then, by Liouville's theorem, 
it suffices to show
\[ \sum_{|n| \leq N} \bigg[ \frac{\dd \Re X_n}{\dd \Re v_n}
+ \frac{\dd \Im X_n}{\dd \Im v_n}\bigg] = 0,
\]

\noi
or equivalently,
\begin{align} \sum_{|n| \leq N} \bigg[ \frac{\dd  X_n}{\dd  v_n}
+ \frac{\dd \cj X_n}{\dd \cj v_n}\bigg] = 0.
\label{meas1a}
\end{align}

\noi
Note that the first sum in \eqref{NLS7} does not 
have any contribution
to \eqref{meas1a} due to the frequency restriction $n_1, n_3 \ne n$.
Hence, we have 
\begin{align*} 
 \frac{\dd  X_n}{\dd  v_n}
+ \frac{\dd \cj X_n}{\dd \cj v_n}
= 2 i |v_n|^2 -2 i |v_n|^2   = 0
\end{align*}

\noi
for each $|n|\leq N$.
Therefore, \eqref{meas1a} holds.
\end{proof}

%
%
%
%
%
%

We now present the proof of 
Proposition \ref{PROP:meas1}.

\begin{proof}[Proof of Proposition \ref{PROP:meas1}]
The first equality in \eqref{meas1} is nothing but the definition of $\rho_{s, N, r, t}$.
From \eqref{Gibbs0} and  \eqref{Gibbs1} with \eqref{Hamil2}, 
we have 
\begin{align*}
\rho_{s, N, r, t}& (\Psi_N(t, \tau)(A))
 = \ft Z_{s, N, r}^{-1}
\int_{E_N} \int_{E_N^\perp}\ind_{\Psi_N(t, \tau)(A)} (v)\ind_{\{ \| v\|_{L^2 } \leq r\}} 
e^{-\frac 12 E_t( \P_{\leq N} v)} dL_N \otimes d\mu_{s, N}^\perp
\notag \\
\intertext{By Fubini's theorem, Lemma \ref{LEM:meas1a}, 
and \eqref{flow1}
we have}
& = \ft Z_{s, N, r}^{-1}
 \int_{E_N^\perp}
\bigg\{ \int_{E_N}
 \ind_{\Psi_N(t, \tau)(A)} ( \wt \Psi_N(t, \tau)( \P_{\leq N} v) + \P_{>N} v)
  \notag \\
 & \hphantom{XXXXXXXX}
 \times \ind_{\{ \|\wt \Psi_N(t, \tau) (\P_{\leq N} v) + \P_{>N} v\|_{L^2 } \leq r\}} 
e^{-\frac 12 E_t( \wt \Psi_N(t, \tau)(\P_{\leq N} v))} dL_N\bigg\} d\mu_{s, N}^\perp
\notag \\
& = \ft Z_{s, N, r}^{-1}
 \int_{E_N^\perp}
\bigg\{ \int_{E_N}
 \ind_{\Psi_N(t, \tau)A} ( \Psi_N(t, \tau)  (v)  )  
  \notag \\
 & \hphantom{XXXXXXXX}
 \times
 \ind_{\{ \|\Psi_N(t, \tau)  v\|_{L^2 } \leq r\}} 
e^{-\frac 12 E_t( \P_{\leq N} \Psi_N(t, \tau) (v))} dL_N\bigg\} d\mu_{s, N}^\perp
\end{align*}

\noi
By the bijectivity of $\Psi_N(t, \tau)$, we have 
$ \ind_{\Psi_N(t, \tau)(A)} ( \Psi_N(t, \tau)  (v)  )   =  \ind_{A} (  v  )  $.
We also have 
the $L^2$-conservation: $\|\Psi_N(t, \tau) ( v)\|_{L^2 } = \| v\|_{L^2 }$
Hence, we have 
\begin{align*}
\rho_{s, N, r, t} (\Psi_N(t, \tau)(A))
 = \ft Z_{s, N, r}^{-1}
 \int_{L^2}
 \ind_{A} (   v  )  \ind_{\{\|   v\|_{L^2 } \leq r\}} 
e^{-\frac 12 E_t( \P_{\leq N} \Psi_N(t, \tau)( v))} dL_N\otimes d\mu_{s, N}^\perp.
\end{align*}

\noi
This proves the second equality in \eqref{meas1}. 
\end{proof}

\subsection{On the evolution of the truncated measures}\label{SUBSEC:evo}

In this subsection, we establish a growth estimate
on the truncated measure $\rho_{s, N, r, t}$.  
The key ingredients 
are 
the energy estimate (Proposition \ref{PROP:energy}), 
the large deviation estimate (Lemma \ref{LEM:LD}),
and the change-of-variable formula
(Proposition \ref{PROP:meas1})
from 
the previous subsections.

\begin{lemma}\label{LEM:meas2}
Let $s > \frac 34 $. 
There exists $0 \leq \beta < 1$
such that, given  $r > 0$, 
there exists $C >0$ such that 
\begin{align}
\frac{d}{dt} \rho_{s, N, r, t}(\Psi_N(t) (A))
\leq C p^\beta \big\{ \rho_{s, N, r, t} (\Psi_N(t)(A))\big\}^{1-\frac 1p}
\label{meas2}
\end{align}

\noi
for any $p \geq 2$, 
any $N \in \N$, any $t \in \R$, 
and  any measurable set 
$A \subset L^2(\T)$.
Here, $\Psi_N(t) = \Psi_N(t, 0)$ as in \eqref{Psit}.

\end{lemma}

\noi
 As in \cite{TzV1, TzV2, Tzv}, 
the main idea of the proof of Lemma \ref{LEM:meas2}
is to reduce the analysis to that at $t = 0$.

\begin{proof}
Let  $t_0 \in \R$.
By the definition of 
$\Psi(t, \tau)$
and Proposition \ref{PROP:meas1},  we have 
\begin{align*}
\frac{d}{dt}\rho_{s, N, r, t}&  (\Psi_N(t)(A))\bigg|_{t = t_0}\notag\\
& = \frac{d}{dt}\rho_{s, N, r, t_0+t} \big(\Psi_N(t_0+ t, t_0)(\Psi_N(t_0)(A))\big)\bigg|_{t = 0}\notag \\
& = \ft Z_{s, N, r}^{-1}
\frac{d}{dt} \int_{\Psi_N(t_0) (A)} \ind_{\{ \| v\|_{L^2 } \leq r\}} 
e^{-\frac 12 E_{t_0+t}( \P_{\leq N} \Psi_N(t_0+ t, t_0) (v))} d L_N \otimes d\mu_{s, N}^\perp \bigg|_{t = 0} \notag\\
& = - \frac 12 
 \int_{\Psi_N(t_0)( A)} 
\frac{d}{dt} E_{t_0+t}\big( \P_{\leq N} \Psi_N(t_0 + t, t_0) (v)\big)\bigg|_{t = 0}
 d \rho_{s, N, r, t_0}.
\end{align*}

\noi
Hence, by Proposition \ref{PROP:energy}, 
H\"older's inequality, 
and Lemma \ref{LEM:LD}, 
we have 
\begin{align*}
\frac{d}{dt}\rho_{s, N, r, t} (\Psi_N(t)(A))\bigg|_{t = t_0}
& \leq C
\Big\| \| v \|_{L^2}^{4+ \theta} \|v\|_{H^{s - \frac 12 - \eps}}^{2-\theta}
\Big\|_{L^p(\rho_{s, N, r, t_0})}
\big\{  \rho_{s, N, r, t_0}(\Psi_N(t_0) (A))\big\}^{1 - \frac 1p} \notag\\
& \leq C
r^{4+\theta}
\Big\| \|v\|_{H^{s - \frac 12 - \eps}}
\Big\|_{L^{(2-\theta)p}(\rho_{s, N, r, t_0})}^{2-\theta}
\big\{  \rho_{s, N, r, t_0}(\Psi_N(t_0) (A))\big\}^{1 - \frac 1p} \notag\\
& \leq C_r
p^{1- \frac{\theta}{2}}
\big\{  \rho_{s, N, r, t_0}(\Psi_N(t_0) (A))\big\}^{1 - \frac 1p}
\end{align*}

\noi
for some small $\theta > 0$.
This proves \eqref{meas2} with $\be = 1- \frac{\theta}{2}$.
\end{proof}

As a corollary to Lemma \ref{LEM:meas2}, 
we obtain the following control on the truncated measures $\rho_{s, N, r, t}$.

\begin{lemma}\label{LEM:meas3}
Let $s > \frac 34$.
Then, given $t \in \R$, $r > 0$, and $\dl >0$, 
 there exists $C = C(t, r, \dl) > 0$
such that 
\begin{align*}
\rho_{s, N, r, t} (\Psi_N(t) (A)) \leq C\big\{\rho_{s, N, r, t} ( A) \big\}^{1-\dl}
\end{align*}
	
\noi
for 
any $N \in \N$
and  any measurable set 
$A \subset L^2(\T)$.

\end{lemma}

\begin{proof}
As in \cite{TzBBM}, we apply a variant of  Yudovich's argument \cite{Y}.
From  Lemma \ref{LEM:meas2}, we have 
\begin{align}
\frac{d}{dt}\big\{ \rho_{s, N, r, t}(\Psi_N(t) (A))\big\}^\frac{1}{p}
\leq C p^{-\al}
\label{meas3}
\end{align}

\noi
for any $p \geq 2$,
where $\al = 1 - \be > 0$.
Integrating \eqref{meas3}, we have
\begin{align}
 \rho_{s, N, r, t}(\Psi_N(t) (A))
&  \leq 
\big\{ 
(  \rho_{s, N, r, t}(A))^\frac{1}{p} + C t p^{-\al}\big\}^p \notag \\
& = 
\rho_{s, N, r, t}(A) e^{p \log \big\{1 +  C t p^{-\al}
\rho_{s, N, r, t}(A)^{-\frac 1 p}\big\}}\notag \\
& 
\leq 
\rho_{s, N, r, t}(A) e^{ C t p^{1-\al}
\rho_{s, N, r, t}(A)^{-\frac 1 p}}, 
\label{meas4}
\end{align}

\noi
where, in the last inequality, we used the fact that  $\log (1+x) \leq x$
for $x \geq 0$.
By choosing $p = 2 - \log \rho_{s, N, r, t}(A)$ such that 
\[ \rho_{s, N, r, t}(A)^{-\frac 1 p} =  e^{ 1- \frac 2 p } \leq e, \]

\noi
it follows from \eqref{meas4} that 
\begin{align}
 \rho_{s, N, r, t}(\Psi_N(t) (A))
\leq 
\rho_{s, N, r, t}(A) e^{ C e t \{ 2 - \log \rho_{s, N, r, t}(A)\}^{1-\al}}.
\label{meas4a}
\end{align}

We claim that, given $\dl > 0$, there exists 
$C = C(t,  \dl, \al) > 0$
such that 
\begin{align}
e^{ C e t \{ 2 - \log \rho\}^{1-\al}} \leq 
 C(t,  \dl, \al) \rho^{-\dl}
\label{meas4b}
\end{align}

\noi
for all $\rho\in [0, 1]$.
By rewriting \eqref{meas4b}, 
it suffices to prove 
\begin{align}
 \{ 2 - \log \rho\}^{1-\al} \leq 
-\dl \log \rho
+ \log C(t,  \dl, \al) 
\label{meas4c}
\end{align}

\noi
Clearly, \eqref{meas4c} holds  as $\rho \to 1-$
by choosing sufficiently large $C(t,  \dl, \al) >0$.
On the other hand, 
 \eqref{meas4c} also holds  as $\rho \to 0+$,
 since $\al > 0$.
Hence, \eqref{meas4c} holds
for all $\rho \in [0, 1]$
by the continuity of $\log \rho$ and choosing sufficiently large $C(t,  \dl, \al) >0$.

Therefore, 
from \eqref{meas4a} and \eqref{meas4b}, 
we conclude that given $\dl > 0$,  
there exists 
$C = C(t, r, \dl, \al) > 0$
such that 
\begin{align*}
\rho_{s, N, r, t} (\Psi_N(t) (A)) \leq C(t, r, \dl, \al) \big\{\rho_{s, N, r, t} ( A) \big\}^{1-\dl}.
\end{align*}

\noi
This completes the proof of Lemma \ref{LEM:meas3}.
\end{proof}

\subsection{Proof of Theorem \ref{THM:quasi}} \label{SUBSEC:proof}

We conclude this section 
by  presenting the proof of Theorem~\ref{THM:quasi} for $s > \frac 34$.
Before doing so, we first upgrade Lemma \ref{LEM:meas3}
to the untruncated measure $\rho_{s, r, t}$.

\begin{lemma}\label{LEM:meas4}
Let $s > \frac 34$.
Then, given  $t \in \R$, $r > 0$, $R >0$, and $\dl >0$, 
 there exists $C = C(t, r, \dl) > 0$
such that 
\begin{align}
\rho_{s,  r, t} (\Psi(t) (A)) \leq C\big\{\rho_{s,  r, t} ( A) \big\}^{1-\dl}.
\label{meas5}
\end{align}

\noi
for  any measurable set 
$A  \subset L^2(\T)$.

\end{lemma}

\begin{proof}

Given $R > 0$, let $B_R$ denote the ball of radius $R$
centered at the origin in $L^2(\T)$.
We first consider the case when $A$ is compact in $L^2$
and $A\subset B_R$ 
for some $R >0$.
It follows 
from Proposition \ref{PROP:approx} and 
Corollary \ref{COR:Gibbs}
that, given $\eps, \g  > 0$, there exists $N_0  = N_0(t, R, \eps, \g)\in \N$
such that 
\begin{align*}
\rho_{s, r, t} (\Psi(t) (A)) 
& \leq 
\rho_{s, r, t} (\Psi_N(t) (A+B_\eps)) 
 \leq \rho_{s, N, r, t} (\Psi_N(t) (A+B_\eps)) + \g
\end{align*}

\noi
for any $N \geq N_0$.
Then, by Lemma \ref{LEM:meas3}
and Corollary \ref{COR:Gibbs}, we have  
\begin{align}
\rho_{s, r, t} (\Psi(t) (A)) 
& \leq 
C(t, r, \dl)
\big\{\rho_{s, N, r, t} (A+B_\eps)  \big\}^{1-\dl} + \g \notag \\
& \leq 
C(t, r,  \dl)
\big\{\rho_{s,  r, t} ( A+B_\eps)  \big\}^{1-\dl} + 2 \g.
\label{meas6} 
\end{align}

\noi
Hence, 
by taking a limit of  \eqref{meas6} as $\eps, \g \to 0$
(with the continuity from above of a probability measure), 
we obtain 
\eqref{meas5} in this case.

Next, let $A$ be any  measurable set in $L^2$.
Then, by the inner regularity of $\rho_{s, r, t}$,
there exists a sequence $\{K_j\}_{j \in \N}$
of compact sets such that $K_j \subset \Psi(t)( A)$ and 
\begin{equation}
 \rho_{s, r, t} (\Psi(t) (A)) = \lim_{j \to \infty} \rho_{s, r, t} (K_j).
\label{meas7}
 \end{equation}

\noi
By the bijectivity of $\Psi(t, \tau)$, we have
\[K_j = \Psi(t, 0) (\Psi(0, t) (K_j))= \Psi(t) (\Psi(0, t) (K_j)).\]

\noi
Note that $\Psi(0, t) (K_j)$ is compact since it is the image of a compact 
set $K_j$ under the continuous map $\Psi(0, t)$.
Moreover, we have $\Psi(0, t) (K_j) \subset \Psi(0, t) \Psi(t) (A) = A$.
Then, by \eqref{meas5} applied to $\Psi(0, t) (K_j)$, we have
\begin{align}
\rho_{s, r, t} (K_j)
& = \rho_{s, r, t} \big(\Psi(t)  (\Psi(0, t) (K_j))\big)
\leq C\big\{\rho_{s,  r, t} ( \Psi(0, t) (K_j) ) \big\}^{1-\dl}\notag \\
& \leq C\big\{\rho_{s,  r, t} ( A) \big\}^{1-\dl}.
\label{meas8}
\end{align}
	
\noi
By taking a limit as $j \to \infty$, 
we obtain 
\eqref{meas5}
from 
\eqref{meas7} and \eqref{meas8}.
\end{proof}

Finally, we present the proof of  Theorem \ref{THM:quasi}.

\begin{proof}[Proof of Theorem \ref{THM:quasi}]

As in Section \ref{SEC:Ramer}, 
it follows from 
 Lemmas \ref{LEM:gauss2}, \ref{LEM:gauss3}, and \ref{LEM:comp}
that it suffices to prove that $\mu_s$
is quasi-invariant under $\Psi(t)$, i.e.
the dynamics of \eqref{NLS5}.

Fix $t \in \R$. Let $A \subset L^2(\T)$ be a measurable set 
such that $\mu_s(A) = 0$.
Then, for any $r > 0$, we have 
\[\mu_{s, r}(A) = 0.\]

\noi
By the mutual absolute continuity 
 of $\mu_{s, r}$ and $\rho_{s, r, t}$, 
 we obtain
\[\rho_{s, r, t}(A) = 0\]

\noi
 for any $r > 0$.
Then, by Lemma \ref{LEM:meas4}, we have
\[\rho_{s, r, t}(\Psi(t) (A)) = 0.\]

\noi
By invoking the mutual absolute continuity 
 of $\mu_{s, r}$ and $\rho_{s, r, t}$ once again, 
 we have
\[\mu_{s, r}(\Psi(t) (A)) = 0.\]

\noi
Then, the dominated convergence theorem yields
\[\mu_{s}\big(\Psi(t) (A)\big) 
= \lim_{r \to \infty} 
\mu_{s, r}\big(\Psi(t) (A)\big) = 0.\]

\noi
This completes the proof of Theorem \ref{THM:quasi}.
\end{proof}

%
%
%
%
%
%
%


\appendix

\section{On the Cauchy problem \eqref{NLS1}}
\label{SEC:WP}

In this appendix, we discuss the well-posedness issue 
for the Cauchy problem \eqref{NLS1}.
In particular, we prove global well-posedness of \eqref{NLS1} in $L^2(\T)$
and ill-posedness below $L^2(\T)$.


\subsection{Well-posedness in $L^2(\T)$}
\label{SUBSEC:LWP}

We say $u$ is a solution to \eqref{NLS1} if $u$ satisfies the following Duhamel formulation
\begin{align*}
u (t)  = S(t) u_0 \mp i \int_0^t S(t - t') |u|^2 u(t') dt', 
\end{align*}

\noi
where $S(t) = e^{-i t\dx^4}$.
The main result of this section is the following local well-posedness of \eqref{NLS1}.

\begin{proposition}\label{PROP:LWP}
Let $s \geq 0$.  Then, given $ u_0\in H^s(\T)$, 
there exist $T = T(\|u_0\|_{L^2}) > 0$ and 
a unique solution $u \in C([-T, T]; H^s)$ to \eqref{NLS1} with $u|_{t = 0} = u_0$.
Moreover, we have
\begin{align}
\sup_{t \in [-T, T]}\|u(t) \|_{H^s} 
& \leq C \| u_0 \|_{H^s}. 
\label{LWP1a}
\end{align}

\end{proposition}

\noi
See Remark \ref{REM:uniq} below for the precise uniqueness statement.
Once we prove Proposition \ref{PROP:LWP}, 
global well-posedness (Proposition \ref{PROP:GWP}) 
  follows from the conservation of mass \eqref{mass1}.
We prove  Proposition \ref{PROP:LWP} via
the Fourier restriction norm method \cite{BO93}.
While the argument is standard, 
we present the details of the proof  for the sake of completeness.

Given $s, b \in \R$, define $X^{s, b}$
as the completion of $\S(\T\times \R)$ under the following norm:
\begin{align}
\| u\|_{X^{s, b}(\T\times\R)} = \| \jb{n}^s \jb{\tau + n^4}^b \ft u(n, \tau)\|_{\l^2_nL^2_\tau}.
\label{Xsb}
\end{align}

\noi
Given a time interval $I \subset \R$, we define
 the local-in-time version $X^{s, b}_I$ restricted to the time interval $I$
 by setting
 \[ \| u \|_{X^{s, b}(I)} = \inf \big\{ \|\wt u  \|_{X^{s, b}} : \, \wt u|_I = u\big\}.\]

\begin{remark}\label{REM:uniq} \rm 
When $s > \frac 12$, the uniqueness statement in Proposition \ref{PROP:LWP}
holds in $C([-T, T]; H^s)$.
When $s \in [0, \frac 12]$, the uniqueness holds
only within a  ball in 
$C([-T, T]; H^s) \cap X^{s, b}([-T, T])$
for some $b > \frac 12$.\footnote{
When $s \geq \frac 16$, 
one can also prove unconditional uniqueness in 
the entire  $C([-T, T]; H^s)$,
 by applying normal form reductions infinitely many times
as in \cite{GKO}.
 The proof is precisely the same as that in \cite{GKO}
 for the standard cubic NLS on $\T$.}
\end{remark}

Before presenting the proof of Proposition \ref{PROP:LWP}, 
we first go over preliminary lemmas.
Let $\eta \in C^\infty_c(\R)$ be a smooth cut off function 
such that $\eta(t) \equiv 1 $ for $|t|\leq 1$ and 
$\eta(t) \equiv 0 $ for $|t|\geq 2$.
Given $T>0$, set $\eta_{_T}(t) = \eta(T^{-1} t)$.
Then, we have the following basic linear estimates.
See \cite{BO93, KPV93, GTV,  TAO}

\begin{lemma}\label{LEM:lin1}
Let $s \in \R$.

\smallskip

\noi
\textup{(i)} For any $b \in \R$, we have
\begin{align*}
\| \eta (t) S(t) u_0 \|_{X^{s, b}}
\leq C_b \|u_0\|_{H^{s}}.
\end{align*}

\smallskip

\noi
\textup{(ii)}
Let $ - \frac 12 < b' \leq 0 \leq b \leq b'+1$.
Then, for $T \leq 1$, we have 
\begin{align*}
\bigg\| \eta_{_T} (t) \int_0^t S(t-t')F(t') dt'\bigg\|_{X^{s, b}}
\leq C_{b, b'} T^{1-b+b'} \|F\|_{X^{s, b'}}.
\end{align*}
\end{lemma}

Next, we state  the $L^4$-Strichartz estimate.
\begin{lemma}\label{LEM:L4}
The following estimate holds:
\begin{align}
\| u \|_{L^4(\T\times \R)} \les \| u \|_{X^{0, \frac{5}{16}}}.
\label{L4}
\end{align}
	
\end{lemma}

\noi
Note that the value $b = \frac 5{16}$ in 
\eqref{L4} is sharp
in the sense that the estimate \eqref{L4}
fails for $b < \frac 5{16}$.\footnote{
Consider the function
\[ u_N(x, t) = \sum_{|n|\leq N} \int_{|\tau|\leq N^4} e^{i(nx + \tau t)} d\tau.\]

\noi
Namely, we have 
$\ft{u}_N(n, \tau) = \ind_N(n)\ind_{N^4}(\tau)$,
where $\ind_N$ is the characteristic function of the interval $[-N, N]$.
Then, a direct computation shows that $\| u_N\|_{L^4(\T\times\R)} \sim N^\frac{15}{4}$, 
while $\| u_N \|_{X^{0, b}}\sim N^{\frac{5}{2} + 4b}$,
showing the sharpness of \eqref{L4}.
}

\begin{proof}
We closely follow the argument for the $L^4$-Strichartz estimate
for the usual (second order) Schr\"odinger equation presented in \cite{TAO}.
Given dyadic $M \geq 1$, let $u_M$ be the restriction of $u$
onto the modulation size $\jb{\tau + n^4} \sim M$.
Then, 
it suffices to show that 
there exists $\eps > 0$ such that 
\begin{align}
\| u_M u_{2^m M} \|_{L^2_{x, t}}
\les 2^{-\eps m } M^\frac{5}{16}\|u_M\|_{L^2_{x, t}}
( 2^mM)^\frac{5}{16}\|u_{2^mM}\|_{L^2_{x, t}}
\label{L4a}
\end{align}

\noi
for any $M \in \N$ and $m \in \N \cup \{0\}$.

Indeed, assuming \eqref{L4a}, by Cauchy-Schwarz inequality, we have
\begin{align*}
\| u \|_{L^4(\T\times \R)}^2
& \les \sum_{M } \sum_{m \geq 0} 
\| u_M u_{2^m M} \|_{L^2_{x, t}}\notag\\
& \les 
 \sum_{M } \sum_{m \geq 0} 
2^{-\eps m } M^\frac{5}{16}\|u_M\|_{L^2_{x, t}}
( 2^mM)^\frac{5}{16}\|u_{2^mM}\|_{L^2_{x, t}}  \notag\\
& \les 
 \sum_{m \geq 0} 
2^{-\eps m } 
\bigg( \sum_{M }
M^\frac{5}{8}\|u_M\|_{L^2_{x, t}}^2\bigg)^\frac{1}{2}
\bigg(
 \sum_{M }( 2^mM)^\frac{5}{8}\|u_{2^mM}\|_{L^2_{x, t}}^2\bigg)^\frac{1}{2}\\
 & \les
 \| u \|_{X^{0, \frac{5}{16}}}^2.
\end{align*}
	
\noi
This proves \eqref{L4}.

Hence, it remains to prove \eqref{L4a}.
By Plancherel's identity and H\"older's inequality,  we have
\begin{align}
\text{LHS of }\eqref{L4a}
& = \bigg\|\sum_{n = n_1 + n_2} \intt_{\tau = \tau_1 + \tau_2}
\ft {u_M}(n_1, \tau_1) \ft {u_{2^m M}} (n_2, \tau_2) d\tau_1\bigg\|_{\l^2_n L^2_\tau}\notag\\
& \leq 
\sup_{n, \tau} A(n, \tau)^\frac{1}{2}
\cdot 
\|u_M\|_{L^2_{x, t}}
\|u_{2^mM}\|_{L^2_{x, t}}
\label{L4b}, 
\end{align}

\noi
where $A(n, \tau)$ is defined by  
\[
A(n, \tau) 
= \sum_{n = n_1 + n_2} \intt_{\tau = \tau_1 + \tau_2} 
\ind_{\tau_1 +n_1^4 = O(M), \,
\tau_2 +n_2^4 = O(2^mM)}\, 
d\tau_1 .
\]
	
\noi	
Integrating in $\tau_1$, we have
\begin{align}
A(n, \tau) 
\les M  \sum_{n = n_1 + n_2} 
\ind_{\tau = - n_1^4  - n_2^4 + O(2^mM)}.
\label{L4c}
\end{align}

\noi
Under $n = n_1 + n_2$ and 
$\tau = - n_1^4  - n_2^4 + O(2^mM)$,  
we have
\begin{align*}
\big( (n_1 - n_2)^2 + 3 n^2 \big)^2 = 8 (n_1^4 + n_2^4) + 8 n^4
= -8 \tau + 8 n^4 +  O(2^mM).
\end{align*}

\noi
This implies that 
$(n_1 - n_2)^2 + 3 n^2$ belongs to at most two intervals of size $ O(2^\frac{m}{2}M^\frac{1}{2})$, 
i.e.
\[(n_1 - n_2)^2 + 3 n^2 = C_{j, \tau, n} + O(2^\frac{m}{2}M^\frac{1}{2})\]

\noi
for some $C_{j, \tau, n}$, $j = 1, 2$.
This, in turn, implies that $n_1-n_2$ belongs
to at most four intervals of size $O(2^\frac{m}{4}M^\frac{1}{4})$.
Hence, from \eqref{L4c}, we have
\begin{align}
A(n, \tau) ^\frac{1}{2}
\les 2^\frac{m}{8}M^\frac{5}{8}
\leq  2^{-\frac{3}{16}m}M^{\frac{5}{16}} (2^m M)^{\frac{5}{16}}.
\label{L4d}
\end{align}

\noi
Finally, \eqref{L4a} follows from \eqref{L4b} and \eqref{L4d}.
\end{proof}

Now, we are ready to prove Proposition \ref{PROP:LWP}.

\begin{proof}[Proof of Proposition \ref{PROP:LWP}]


Let $u_0 \in L^2(\T)$.
Given $0< T  \leq 1$, let 
\begin{align}
\G(u) (t)  = \G_{u_0}(u) (t) : = \eta(t) S(t) u_0 \mp i \eta_{_T}(t) \int_0^t S(t - t') |u|^2 u(t') dt'.
\label{NLS3}
\end{align}

\noi
Let $b >\frac 12  $ and small $\dl > 0$.  Then, 	
from Lemma \ref{LEM:lin1}, a duality argument, H\"older's inequality,  and Lemma \ref{LEM:L4}, we have 
\begin{align}
\| \G(u)  \|_{X^{0,b}} 
& \les \| u_0 \|_{L^2} + T^\dl \big\|  |u|^2 u \big\|_{X^{0, b - 1+\dl}} \notag \\
& = \| u_0 \|_{L^2} + T^\dl \sup_{\|v\|_{X^{0, 1-b-\dl} = 1 }} \bigg| \int |u|^2 u \cdot v \, dx dt\bigg|\notag \\
& \leq \| u_0 \|_{L^2} + T^\dl \sup_{\|v\|_{X^{0, 1-b-\dl} = 1 }} \|u\|^3_{L^4_{x, t}} \| v\|_{L^4_{x, t}}\notag \\
& \les \| u_0 \|_{L^2} +  T^\dl  \|u\|^3_{X^{0, \frac 5 {16}}}, 
\label{LWP1}
\end{align}

\noi
as long as 
$ b \leq \frac{11}{16} - \dl$.
Similarly, we have 
\begin{align}
\| \G(u) - \G(v)   \|_{X^{0,b}} 
\les  T^\dl \big( \|u\|^2_{X^{0, b}} +   \|v\|^2_{X^{0, b}}\big)  \|u - v\|_{X^{0, b}}.
\label{LWP2}
\end{align}

\noi
Hence, it follows from \eqref{LWP1} and \eqref{LWP2} that 
$\G$ is a contraction on some ball in $X^{0, b}$
as long as $T = T(\| u_0\|_{L^2}) > 0$ is sufficiently small.

Now, suppose that $u_0 \in H^s(\T)$ for some  $s > 0$. Then, proceeding as in \eqref{LWP1}
with  $T = T(\| u_0\|_{L^2}) > 0$ as above, 
 we have 
\begin{align}
\|\G(u) \|_{X^{s, b}} 
& \les \| u_0 \|_{H^s} +  T^\dl  \|u\|^2_{X^{0, b}}  \|u\|_{X^{s, b}}
\les \| u_0 \|_{H^s} +  T^\dl  \|u_0\|_{L^2}^2  \|u\|_{X^{s, b}}, 
\label{LWP3}
\end{align}

\noi
yielding \eqref{LWP1a}.
A similar argument yields local Lipschitz dependence of the solution map
on $H^s(\T)$.
This completes the proof of Proposition \ref{PROP:LWP}.
\end{proof}

\begin{remark}\label{REM:growth} \rm
When $s = 0$, the conservation of mass
yields $\|u(t) \|_{L^2} = \|u_0 \|_{L^2}  $ for all $t \in \R$.

Now, suppose that $s>0$.
Then, by iterating 
\eqref{LWP1a}
along with the mass conservation,
we conclude that there exists $\theta > 0$ such that 
the following growth estimate on the $H^s$-norm holds:
\begin{align}
\sup_{t \in [0, \tau]}\|u(t) \|_{H^s} 
& \leq C^{K^\theta \tau} \| u_0 \|_{H^s} 
\label{LWP4}
\end{align}

\noi
for any $\tau > 0$ and all $u_0 \in H^s(\T)$ with $\|u_0 \|_{L^2}\leq K$.

\end{remark}

\subsection{Ill-posedness below $L^2(\T)$}
\label{SUBSEC:illposed}

In the following, we briefly discuss the ill-posedness of \eqref{NLS1}
below $L^2(\T)$.
We first present the following failure of uniform continuity of the solution map on bounded sets below $L^2(\T)$.

\begin{lemma}\label{LEM:ill}
Let $s < 0$.
There exist two sequences $\{u_{0, n}\}_{n \in \N}$ and $\{\wt u_{0, n}\}_{n \in \N}$
in $H^\infty(\T)$ 
such that 
\begin{itemize}

\item[\textup{(i)}]  
$u_{0, n}$ and $\wt u_{0, n}$ are uniformly bounded in $H^s(\T)$,

\smallskip

\item[\textup{(ii)}]  
$\displaystyle \lim_{n \to \infty} 
\|u_{0, n}- \wt u_{0, n}\|_{H^s} = 0$,

\smallskip

\item[\textup{(iii)}]  
Let $u_n$ and $\wt u_n$ be the solutions
to \eqref{NLS1} with initial data
$u_n|_{t = 0} =u_{0, n}$
and  
$\wt u_n|_{t = 0} =\wt u_{0, n}$, respectively.
Then, 
there exists $c > 0$ such that 
\[ \liminf_{n \to \infty}
\| u_n - \wt u_n\|_{L^\infty([-T, T]; H^s)} \geq c > 0\]

\noi
for any $T > 0$.

\end{itemize}

\end{lemma}

Lemma \ref{LEM:ill}
exhibits 
 a ``mild'' ill-posedness result for $s< 0$.
The proof of Lemma \ref{LEM:ill}
closely follows the argument in 
Burq-G\'erard-Tzvetkov \cite{BGT1}
and Christ-Colliander-Tao \cite{CCT1}.

\begin{proof}
Given $N \in \N$ and $a \in \C$, 
define $u^{(N, a)}$ by 
\begin{align*}
 u^{(N, a)}(x, t) = N^{-s} a e^{i(Nx - N^4 t \mp N^{-2s} |a|^2 t)}.
 \end{align*}

\noi
Then, it is easy to see that $u^{(N, a)}$ is a smooth global solution to \eqref{NLS1}.

Given $n \in \N$, let
$u_{0, n} = u^{(N_n, 1)}(0)$
and $\wt u_{0, n} = u^{(N_n, 1 + \frac 1 n)}(0)$,
where $N_n \in \N $ is to be chosen later.
Then,  we have 
\begin{align}
\|u_{0, n}\|_{H^s} ,  \|\wt u_{0, n}\|_{H^s} \les 1
\label{ill0}
\end{align}
	
\noi
uniformly in $n \in \N$.
Moreover,  we have
\begin{align}
\|u_{0, n} - \wt u_{0, n}\|_{H^s} \sim \frac 1n.
\label{ill1}
\end{align}

\noi
Note that \eqref{ill0} and \eqref{ill1}
hold independently of a choice of $N_n \in \N$.

Let $u_n$ and $\wt u_n$ be the solutions
to \eqref{NLS1} with initial data
$u_n|_{t = 0} =u_{0, n}$
and  
$\wt u_n|_{t = 0} =\wt u_{0, n}$, respectively.
Namely, 
$u_{ n} = u^{(N_n, 1)}$
and $\wt u_{n} = u^{(N_n, 1 + \frac 1 n)}$.
Given  $n \in \N$, 
define $t_n > 0$ by 
\begin{align*}
t_n = 
\frac{\pi N_n^{2s}}{ \big(1+\tfrac 1n\big)^2 - 1} .
\end{align*}

\noi
Since $s < 0$, we can choose $N_n \in \N$ sufficiently large
such that  $t_n \leq \frac 1 n$.
Then, 
we have
\begin{align}
\|u_{n}(t_n) - \wt u_{n}(t_n)\|_{H^s} 
= \Big|  e^{\mp i N_n ^{-2s}\{1 - (1+\frac 1n)^2\} t_n} - \big(1 + \tfrac 1n\big)\Big|
= 2 + \tfrac 1n \geq 2.
\label{ill3}
\end{align}

\noi
Noting that $t_n \to 0$ as $ n \to \infty$, 
Lemma \ref{LEM:ill} follows from \eqref{ill0}, \eqref{ill1}, and \eqref{ill3}.
\end{proof}

\begin{remark}\rm
The cubic NLS \eqref{cubicNLS} enjoys the Galilean symmetry, 
which preserves the $L^2$-norm.
Namely, $L^2$ is critical with respect to the Galilean symmetry.
Indeed,  the cubic NLS is known to be ill-posed
below $L^2(\T)$.
See \cite{BGT1, CCT1, CCT2, MOLI, GO}.

As for the fourth order NLS \eqref{NLS1}, 
there seems to be no Galilean symmetry\footnote{Here, the Galilean symmetry means basically a translation in the spatial frequency domain
with a certain modulation.
While this modulation is linear in the spatial frequency
for the cubic NLS \eqref{cubicNLS}, 
such a modulation for \eqref{NLS1} is of higher degree 
for \eqref{NLS1} in order to match up with $\dx^4$,
which is inconsistent with the nonlinearity.}
and it is not clear why the regularity $s = 0$ plays a role as a critical value.

\end{remark}

\begin{remark}[non-existence of solutions below $L^2(\T)$]
\rm
The mild ill-posedness of \eqref{NLS1} stated in Lemma \ref{LEM:ill} 
can be updated to the following strong form of ill-posedness
of \eqref{NLS1} below $L^2(\T)$.
Roughly speaking, 
if $u_0 \notin L^2(\T)$, 
then 
there is no weak solution to \eqref{NLS1}.
More precisely,
there exists $s_0 < 0$ such that,  
for $s_0 < s < 0$ and  any $T>0$, 
there exists no weak solution $u \in C([-T, T];H^s(\T))$ to 
 NLS \eqref{NLS1}
such that

\begin{itemize}
\item[\textup{(i)}] $u|_{t = 0} = u_0 \in H^s(\T)\setminus L^2(\T)$

\smallskip

\item[\textup{(ii)}] There exist smooth solutions $\{u_n\}_{n\in \mathbb{N}}$ to \eqref{NLS1} such that 
$u_n \to u$ in $ C([-T, T];H^s(\T))$ as $n \to \infty$. 
\end{itemize}

\noi
Note that this is one of the strongest forms of ill-posedness.

In the following, we present a sketch  of the argument.
See \cite{GO, OW} for details.
Namely, 
first
use the short time Fourier restriction norm method and 
establish an a priori bound in $H^s$, $s< 0$, 
to the renormalized equation \eqref{NLS4}.
Here, the main observation is that 
Lemma \ref{LEM:phase}
guarantees that the a priori bound for the renormalized cubic NLS
also holds for the renormalized fourth order NLS \eqref{NLS4}.\footnote{In fact, 
one can establish an a priori bound for \eqref{NLS4}
for lower regularities. We, however, do not pursue this issue here.
See \cite{OW}.}
This allows us to 
prove an existence result for \eqref{NLS4} in $H^s(\T)$, 
$s_0 < s < 0$, for some $s_0 < 0$.
Recall that if $u$ is a smooth solution to \eqref{NLS1}, 
then 
$\wt u = \mathcal{G}[u]$
 is a smooth solution to \eqref{NLS4}.

Now, let $u_0 \in H^s(\T)\setminus L^2(\T)$, $s \in (s_0, 0)$
and let $\{u_{0, n}\}_{n \in \N} \subset L^2(\T)$
such that $u_{0, n} \to u_0$ in $H^s(\T)$ as $n \to \infty$.
Let $u_n$ denote the unique (global) solution 
to \eqref{NLS1} with $u_n|_{t = 0} = u_{0, n}$
and let $\wt u_n = \mathcal G [u_n]$.
Then, from the a priori bound, there exists $T = T(\|u_0\|_{H^s})>0$
such that 
(i) $\{\wt u_n\}_{n \in \N}$ is bounded in $C([-T, T]; H^s)$
and (ii) $\wt u_n$ converges to some $\wt u$  in $C([-T, T]; H^s)$.
Moreover, $\wt u $ is a solution to \eqref{NLS4}.
In particular, $\wt u(0) = u_0$.
On the other hand, 
in view of  $\| u_n (0)\|_{L^2} \to \infty$ as $n \to \infty$, 
we have faster and faster phase oscillations in \eqref{gauge2},
as $n \to \infty$.
Hence,  $\wt u_n = \mathcal{G}[ u_n]$ converges to 0
in $\mathcal{D}'(\T\times [-T, T])$.
In particular, this implies $\wt u(0) = 0$.
This is clearly a contradiction since $\wt u(0) = u_0 \notin L^2(\T)$.
\end{remark}

\section{On the approximation property of the truncated dynamics}
\label{SEC:approx}

In this appendix, we perform further analysis on 
the equation \eqref{NLS5}
and its truncated approximation \eqref{NLS6}
and establish a certain approximation property.
See Proposition \ref{PROP:approx} below.

Given $N \in \N$, we first consider the following approximation to \eqref{NLS1}:
\begin{align}
\begin{cases}
i \dt u_N =   \dx^4 u_N  -  \P_{\leq N} (|\P_{\leq N}u_N|^{2}\P_{\leq N}u_N) \\
u_N|_{t = 0} = u_0.
\end{cases}
\label{ANLS1}
\end{align}

\noi
We first study the approximation property of \eqref{ANLS1} to \eqref{NLS1}.
By a slight modification of the proof of Proposition \ref{PROP:LWP}, 
it is easy to see that \eqref{ANLS1} is globally well-posed in $L^2(\T)$.

Let $\Phi(t)$ and $\Phi_N(t)$ be the solution maps
to \eqref{NLS1} and \eqref{ANLS1}, respectively.
Given $R > 0$, 
let $B_R $ be the ball of radius $R$ centered at the origin in $L^2(\T)$.
Let  $\tau > 0$.
By iterating the local-in-time argument (see \eqref{LWP3}), 
we have the following uniform estimate:
\begin{align}
\sup_{N\in \N\cup\{\infty\}} \sup_{u_0 \in B_R} 
\|\Phi_N(t)(u_0) \|_{X^{0, b}([0, \tau])}   \leq C(\tau, R)
\label{C1}
\end{align}

\noi
for some $b > \frac 12$, with the understanding that $\Phi_\infty(t) = \Phi(t)$.
Then, from Lemma \ref{LEM:L4}, we obtain 
\begin{align}
\sup_{N\in \N\cup\{\infty\}} \sup_{u_0 \in B_R} 
\|\Phi_N(t)(u_0) \|_{L^4_t([0, \tau]; L^4_x)}   \leq C(\tau, R).
\label{C2}
\end{align}

\begin{lemma}\label{LEM:Bcontrol}
Given $R > 0$, let $A \subset B_R$ be a compact subset in $L^2(\T)$.
Given $\tau > 0$ and $\eps >0$, 
there exists $N_0 \in \N$ such that we have
\begin{align}
\| \P_{>N} \Phi(t)(u_0)\|_{L^4([0, \tau]; L^4_x)} < \eps
\label{Bcontrol1}
\end{align}

\noi
for all $u_0 \in A$ and $N \geq N_0$.
	
\end{lemma}

\begin{proof}
By the continuity of the map: $u_0 \in L^2 \mapsto \Phi(t) u_0 \in 
L^4([0, \tau]; L^4_x)$
and the compactness of $K$, we see that $\Phi(t) K$ is compact
in $L^4([0, \tau]; L^4_x)$.
Hence, there exists a finite index set $\mathcal{J}$ 
and $\{u_{0, j}\}_{j \in \mathcal{J}} \subset K$
such that, given $u_0 \in K$,   we have
\begin{align}
 \| \Phi(t) (u_0) - \Phi(t) (u_{0, j})\|_{L^4([0, \tau]; L^4_x)} <\frac{\eps}{2}
 \label{Bcontrol2}
\end{align}

\noi
for some $j \in \mathcal{J}$.
It follows from \eqref{C2}
and the dominated convergence theorem
that given $ j \in \mathcal{J}$, there exists
$N_j \in \mathbb{N}$ such that 
\begin{align}
\| \P_{>N} \Phi(t)(u_{0, j})\|_{L^4([0, \tau]; L^4_x)} < \frac \eps2
\label{Bcontrol3}
\end{align}

\noi
for all $N \geq N_j$.
Hence, by setting $N_0 = \max_{j \in \mathcal{J}} N_j$, 
\eqref{Bcontrol1} follows from 
\eqref{Bcontrol2} and 
\eqref{Bcontrol3}. 
\end{proof}

We first
establish  the following approximation property of \eqref{ANLS1} to \eqref{NLS1}.

\begin{lemma}\label{LEM:Bapprox}
Given $R > 0$, let $A \subset B_R$ be a compact set in $L^2(\T)$.
Then, for any $\tau > 0$ and $\eps > 0$, 
there exists $N_0 \in \N$ such that 
\begin{align}
\|  \Phi(t)(u_{0}) -  \Phi_N(t)(u_{0})\|_{L^\infty_t([0, \tau]; L^2_x)} < \eps.
\label{Bapprox1}
\end{align}

\noi
for all $u_0 \in A$ and $N \geq N_0$.

\end{lemma}

\begin{proof}
Given $u_0 \in A$, 
let $w_N = u - \P_{\leq N} u_N = \Phi(t) (u_0)  -\P_{\leq N} \Phi_N(t)(u_0) $.
Then, $w_N$ satisfies
\begin{align*}
w_N(t) = - i \int_0^t S(t - t') \QQ (u, \P_{\leq N} u_N)(t') dt'.
\end{align*}
	
\noi
where $\QQ(u, \P_{\leq N}u_N)$ is defined by 
\begin{align*}
\QQ(u, \P_{\leq N}u_N) 
& = \P_{\leq N }\big( |u|^2 u - |\P_{\leq N} u|^2 \P_{\leq N} u\big)
+ 
\P_{> N } (|u|^2 u). 
\end{align*}

Given $T> 0$, let $\wt u$ and $\wt u_N$
be extensions of $u |_{[0, T]}$ and $\P_{\leq N} u_N |_{[0, T]}$ onto $\R$.
Then, defining $\wt w_N$ by 
\begin{align*}
\wt w_N(t) = - i \eta_{_T}(t) \int_0^t S(t - t') \QQ (\wt u, \wt  u_N)(t') dt', 
\end{align*}

\noi
we see that $\wt w_N$ is an extension of $w_N|_{[0, T]}$ onto $\R$.
Given small $\dl > 0$, let $ \frac 12 < b \leq \frac {11}{16}-\dl$ as in the proof of Proposition \ref{PROP:LWP}.
Then, by proceeding as in \eqref{LWP1}, we have
\begin{align}
 \| w_N \|_{X^{0, b} ([0, T])}
& \leq   \|\wt  w_N \|_{X^{0, b} (\T\times \R)}\notag\\
& \leq C T^\dl 
\big( \|\wt  u \|_{X^{0, \frac 5 {16}}}^2 
+ \|\wt  u_N  \|_{X^{0, \frac 5{16}}}^2\big)
\| \wt u -  \wt  u_N  \|_{X^{0, \frac 5{16}}} 
\notag\\
& \hphantom{XXXX}
+ C\| \wt u \|_{L^4_{x, t}}^2 \|\P_{> \frac N 3} \wt u \|_{L^4_{x, t}}
\label{Bapprox2}
\end{align}

\noi
for $T = T(R) > 0$ sufficiently small.
Noting that 
\eqref{Bapprox2} holds for any extensions $\wt u$ and $\wt u_N$
and that  $\wt u - \wt u_N$ is an extension of $w_N|_{[0, T]}$, 
we obtain 
\begin{align}
 \| w_N \|_{X^{0, b} ([0, T])}
& \leq C T^\dl 
\big( \| u \|_{X^{0, \frac 5 {16}}([0, T])}^2 
+ \| u_N  \|_{X^{0, \frac 5{16}}([0, T])}^2\big)
\| w_N \|_{X^{0, \frac 5{16}}([0, T])} \notag\\
& \hphantom{X}
+ C\| u \|_{L^4_t([0, T]; L^4_x)}^2 \|\P_{> \frac N 3}  u \|_{L^4_t([0, T]; L^4_x)}.
\label{Bapprox3}
\end{align}

\noi
By making $T = T(\tau, R)> 0$ sufficiently small, 
it follows
from \eqref{Bapprox3} with \eqref{C1}, 
that 
\begin{align*}
\| w_N \|_{L^\infty([0, T]; L^2)}
& \leq C  \| w_N \|_{X^{0, b} ([0, T])}
\leq C\| u \|_{L^4([0, T]; L^4_x)}^2 \|\P_{> \frac N 3}  u \|_{L^4_t([0, T]; L^4_x)}.
\end{align*}

\noi
Hence, by Lemma \ref{LEM:Bcontrol} with \eqref{C2}, 
\begin{align*}
\| w_N \|_{L^\infty([0, T]; L^2)}
 = o_N(1)
\end{align*}

\noi
as $N \to \infty$, uniformly in $u_0 \in A$.

By repeating the argument, we obtain 
\begin{align*}
 \| w_N \|_{X^{0, b} ([T, 2T])}
& \leq  o_N(1) + C T^\dl 
\big( \| u \|_{X^{0, \frac 5 {16}}([T, 2T])}^2 
+ \| u_N  \|_{X^{0, \frac 5{16}}([T, 2T])}^2\big)
\| w_N \|_{X^{0, \frac 5{16}}([T, 2T])} \notag\\
& \hphantom{X}
+ C\| u \|_{L^4([T, 2T]; L^4_x)}^2 \|\P_{> \frac N 3}  u \|_{L^4([T, 2T]; L^4_x)}.
\end{align*}

\noi
As before, this in turn implies
\begin{align*}
\| w_N \|_{L^\infty_t([T, 2T]; L^2_x)}
 = o_N(1)
\end{align*}

\noi
as $N \to \infty$, uniformly in $u_0 \in A$.
By arguing iteratively
on time intervals of length $T$, we can cover the whole time interval $[0, \tau]$
and we conclude that there exists $N_1 = N_1(\tau, \eps, R) \in \N$ such that 
\begin{align}
\|  \Phi(t) (u_0)  -\P_{\leq N} \Phi_N(t)(u_0)  \|_{L_t^\infty([0, \tau]; L^2_x)}
< \frac{\eps}{2}
 \label{Bapprox4}
\end{align}

\noi
for all $u_0 \in A$ and $N \geq N_1$.

It remains to control 
$\P_{> N} \Phi_N(t)(u_0) $.
Recall that the solution map $\Phi_N(t)$
to \eqref{ANLS1} is locally uniformly continuous.
Moreover, 
it follows from a slight modification of the proof of Proposition \ref{PROP:LWP}
that 
the modulus of continuity is
uniform in $N \in \N$. 
Hence, 
for any $\eps > 0$, there exists $\dl > 0$
such that 
if $u_0, u_1 \in B_R$ satisfies 
$\| u_0 - u_1\|_{L^2} < \dl $, 
then we have 
\[ \| \Phi_N(t) (u_0) - \Phi_N(t) (u_1)\|_{L^\infty_t([0, \tau]; L^2_x)} < \frac \eps 4\]

\noi
for all $N \in \N$.
By the compactness of $A$, we can cover $A$ by  finitely many 
ball of radius $\dl$ centered at $u_{0, j}$, $j = 1, \dots, J$
for some $J < \infty$
such that, given $u_0 \in A$,
there exists  $j \in \{1, \dots, J\}$ such that 
\begin{align}
 \| \Phi_N(t) (u_0) - \Phi_N(t) (u_{0, j})\|_{L^\infty_t([0, \tau]; L^2_x)} < \frac \eps 4
\label{Bapprox5}
\end{align}

\noi
for all $N \in \N$.

Noting that 
$\P_{> N} \Phi_N(t)(u_{0, j})  = S(t) \P_{> N} u_{0, j}$, 
there exists $N_2 \in \N$ such that 
\begin{align}
 \|\P_{> N} \Phi_N(t)(u_{0, j}) \|_{L^\infty_t([0, \tau]; L^2_x)} 
 =  \| S(t) \P_{> N} u_{0, j}\|_{L^\infty_t([0, \tau]; L^2_x)} 
=  \|  \P_{> N} u_{0, j}\|_{L^2_x}  
 < \frac \eps 4
\label{Bapprox6}
\end{align}

\noi
for all $N \geq N_2$ and $j = 1, \dots, J$.

Therefore, the desired estimate \eqref{Bapprox1} 
follows from 
\eqref{Bapprox4}, 
\eqref{Bapprox5}, 
and \eqref{Bapprox6}.
\end{proof}

Recall that, if $u$ is a solution to \eqref{NLS1}, 
then  
$v(t) = S(-t) \GG[u](t)$
is a solution to \eqref{NLS5}, 
where the gauge transformation $\GG$ is defined in \eqref{gauge2}.
Noting that  the truncated $L^2$-norm
$\int |\P_{\leq N} u|^2dx $ is conserved for \eqref{ANLS1}, 
define 
 $\GG_N$ by 
\begin{align}
 \GG_N[u](t) = e^{2i t \fint |\P_{\leq N} u|^2 }  u 
\label{B1}
 \end{align}

\noi
for a solution $u$ to \eqref{ANLS1}.
Then, letting 
\[ v  = S(-t) \GG_N[u](t),\] 

\noi
we see that $v$ is a solution to \eqref{NLS6}.
Recalling that $\Psi(t)= \Psi(t, 0)$ and $\Psi_N(t)= \Psi_N(t, 0)$ represent
the solution maps to \eqref{NLS5} and \eqref{NLS6}, respectively, 
we have
\begin{align}
\Psi(t) (u_0) = S(-t) \circ \GG\circ \Phi(t)(u_0)
\qquad \text{and}
\qquad 
\Psi_N(t) (u_0) = S(-t)\circ  \GG_N\circ \Phi_N(t)(u_0).
\label{B2}
\end{align}

We conclude this appendix by 
establishing the following approximation property
of \eqref{NLS6} to \eqref{NLS5}.
Lemma \ref{LEM:Bapprox}
and \eqref{B2} play an important role.

\begin{proposition}\label{PROP:approx}

Given $R>0$, let $A \subset B_R$ be a compact set in $L^2(\T)$.
Fix $t \in \R$.
Then, for any $\eps > 0$, there exists $N_0 = N_0(t, R, \eps) \in \N$ such that we have
\begin{align} 
\Psi(t) (A)\subset \Psi_N(t) (A + B_\eps)
\label{B3}
\end{align}

\noi
for all $N \geq N_0$.

\end{proposition}

\begin{proof}
By writing $\Psi(t)(A)$ as 
\[ \Psi(t) (A) = \Psi_N(t)\big( \Psi_N(0, t) \Psi(t) (A)\big),\]

\noi
it suffices to show 
$ \Psi_N(0, t) \Psi(t) (A) \subset A + B_\eps$.
Given $w_N \in  \Psi_N(0, t) \Psi(t) (A)$, we have
$w_N =  \Psi_N(0, t) \Psi(t) (u_0)$ for $u_0 \in A$.
Thus, we can rewrite $w_N$ as 
$w_N =  u_0+ z_N$, 
where 
\[z_N : = \Psi_N(0, t)\big( \Psi(t)( u_0) -  \Psi_N(t)( u_0)\big).\]

By the unitarity of $\Psi_N(0, t)$, \eqref{B2}, 
and the unitarity of $S(-t)$, 
we have 
\begin{align*}
\|z_N \|_{L^2}
& = \| \Psi(t) (u_0) -  \Psi_N(t) (u_0)\|_{L^2}\\
& = \| \GG\circ \Phi(t) (u_0) -  \GG_N\circ\Phi_N(t) (u_0)\|_{L^2}.
\end{align*}
	
\noi
By the mean value theorem with 
\eqref{gauge2} and \eqref{B1}
followed by 
Lemma \ref{LEM:Bapprox}
and the unitarity of $\Phi(t)$,  we have 
\begin{align*}
\|z_N \|_{L^2}
& \leq 
\|  \Phi(t) (u_0) - \Phi_N(t) (u_0)\|_{L^2}
+ 
C \big( \| u_0 \|_{L^2}^2  - \| \P_{\leq N}u_0 \|_{L^2}^2 \big)\| \Phi(t) u_0\|_{L^2}\notag\\
& 
< \frac \eps 2 + 
C R^2 \| \P_{> N }u_0 \|_{L^2}
< \eps
\end{align*}

\noi
for all sufficiently large $N \gg1$, 
uniformly in $u_0 \in A \subset B_R$.
This proves \eqref{B3}.
\end{proof}

\begin{ackno}\rm
T.O.~was supported by the European Research Council (grant no.~637995 ``ProbDynDispEq'').
N.T.~was supported by the European Research Council (grant no.~257293 ``DISPEQ'').
T.O.~would like to thank Universit\'e de Cergy-Pontoise 
for its hospitality during his visit.
The authors are grateful to the anonymous referees
for their helpful comments that have improved the presentation of this paper.
\end{ackno}

\end{document}